\documentclass{article}
\usepackage{amsmath}
\usepackage{amsfonts}
\usepackage{amsthm}
\usepackage{enumerate}

\newtheorem{thm}{Theorem}
\newtheorem{prop}[thm]{Proposition}
\newtheorem{lem}[thm]{Lemma}
\newtheorem{cor}[thm]{Corollary}

\theoremstyle{definition}
\newtheorem{defn}[thm]{Definition}
\newtheorem{rmk}[thm]{Remark}

\newcommand{\sqbinom}[2]{\big[\begin{smallmatrix} #1 \\ #2 \end{smallmatrix}\big]}

\begin{document}

\title{The Combinatorics of Higher Derivatives of Implicit Functions}

\author{Shaul Zemel}

\maketitle

\section*{Introduction}

Several formulae for derivatives, in various settings, are taught to first year undergraduates: Leibniz's rule for differentiating products, the chain rule for the derivative of a composition of functions, the derivative of an inverse function, and the derivative of an implicit function. A natural question arising from these formulae is whether one can obtain explicit expressions for derivatives of higher orders in these settings. In the first case the answer is given in terms of the generalized Leibniz rule, involving just simple binomial coefficients, as in the usual Binomial Theorem. A further generalization, in which one differentiates the product of more than two functions of the same variable, produces a similar analogue of the Multinomial Theorem.

The answer in the second case, for the derivative of some order $n$ of the composition $z=z\big(y(x)\big)$ with respect to $x$, requires partitions of $n$. A partition $\lambda$ of $n$, denoted by $\lambda \vdash n$, is usually represented by a decreasing sequence of positive integers $a_{q}$, $1 \leq q \leq p$, with $\sum_{q=1}^{p}a_{q}=n$. However, an alternative description is given in terms of multiplicities: The \emph{multiplicity} $m_{j}$ of $j$ in $\lambda$ is the (non-negative) number $\big|\{1 \leq q \leq p|a_{q}=j\}\big|$ counting how many times the integer $j$ appears in the partition $\lambda$, and we have $\sum_{j=1}^{t}jm_{j}=n$. The \emph{length} $\ell(\lambda)$ is just the maximal index $p$, which also equals $\sum_{j=1}^{t}m_{j}$. By using $f^{(r)}(t)$ for the $r$th derivative of $f$ with respect to $t$, the Fa\`{a} di Bruno formula reads
\[\frac{d^{n}z}{dx^{n}}=\sum_{\lambda=\{m_{j}\}_{j=1}^{n} \vdash n}n!z^{(\ell(\lambda))}\big(y(x)\big)\prod_{j=1}^{n}\frac{y^{(j)}(x)^{m_{j}}}{j!^{m_{j}}m_{j}!}.\] Note that while the derivatives themselves can be described in terms of the sequence $\{a_{q}\}_{q=1}^{p}$ as $z^{(q)}\big(y(x)\big)$ and $\prod_{p=1}^{q}y^{(a_{p})}(x)$, the combinatorial coefficient $n!\big/\prod_{j=1}^{n}j!^{m_{j}}m_{j}!$ associated with $\lambda$ requires the presentation with the multiplicities. For the history of Fa\`{a} di Bruno's formula, as well as several of its proofs, see \cite{[J1]} and the references therein.

The third case is implicitly solved by the Lagrange Inversion Formula, but it seems that a closed expression, not involving power series, has appeared for the first time in \cite{[J2]} (see, on the other hand, the historical discussion in that paper and its references). If $x=g(y)$ and $n\geq2$ then the expression is
\[\frac{d^{n}}{dx^{n}}g^{-1}(x)=\sum_{u=1}^{n-1}\sum_{\{\mu_{j}\}_{j\geq2} \in M_{n,u}}\frac{(n+u-1)!(-1)^{u}}{\prod_{j}j!^{\mu_{j}}\mu_{j}!}\frac{\prod_{j}g^{(j)}(y)^{\mu_{j}}}{g'(y)^{n+u}},\] where $M_{n,u}$ is the set of sequences $\{\mu_{j}\}_{j\geq2}$ of non-negative integers that satisfy the equality $\sum_{j}\mu_{j}=u$ and $\sum_{j}j\mu_{j}=u+n-1$, namely the set of partitions of the number $n+u-1$ as the sum of $u$ numbers that are at least 2. This explains why $\mu_{j}$ can be non-zero only for $j \leq n$, and also why only $1 \leq u \leq n-1$ is allowed, though for the excluded case $n=1$ we do have a single term with $u=0$. The combinatorial coefficient is the number of possibilities of dividing $n+u-1$ marked balls into $u$ identical boxes according to the partition $\{\mu_{j}\}_{j}$.

The fourth question was partially addressed in \cite{[Wo]}, whose formula is not entirely explicit, and a full answer was established much later. The first reference known to the author to give a closed formula for the general derivative of an implicit function, based on some variant of the Bell polynomials, is \cite{[C1]}, essentially using the Lagrange Inversion Formula. The problem is briefly discussed on page 153 of \cite{[C2]}, where the Lagrange Inversion Formula and its extensions are still the basis of the discussion (this is also the case in most of the references cited there, except \cite{[Wo]}). A more explicit expression appears in \cite{[CF]}, though the formula in that reference contains an error (mainly in the argument for counting the number of elements appearing in the formula), and this mistake is corrected in the (comparably very recent) reference \cite{[Wi]}. The latter reference uses an integration argument combined with the Fa\`{a} di Bruno formula, but also gives an inductive proof, resembling our proof in spirit. We also mention \cite{[N]}, which determines the terms appearing in such a derivative, but without the exact coefficients. The pre-print \cite{[J3]} gives the full answer as well, but with the coefficients appearing as combinatorial descriptions rather than closed formulae. The case $m=r=1$ in \cite{[STZ]} essentially evaluates higher derivatives of implicit functions as well (using their Taylor expansions with $x_{0}=y_{0}=0$), in terms of sums over certain trees. For another approach see \cite{[S]}, which based on some ideas of \cite{[Y]} provides formulae that are related to ours, though from a different point of view.

In this paper we give another type of a closed formula for the $n$th derivative $y^{(n)}(x)$ with $n\geq2$, when $y$ is the function of $x$ given implicitly via the equation $f(x,y)=0$. The formula that we prove is based not on the most elementary products of the various partial derivatives of $f$, but on products of certain binomial combinations. This means that the sum describing $y^{(n)}$ contains substantially less terms, and the coefficients have interesting combinatorial meaning, similar to those from \cite{[Wi]} and \cite{[J3]}. In the end we show how to deduce the formula appearing in those references from ours. The formula for higher derivatives of inverse functions, from \cite{[J2]} and others, follows as a special case.

It is also known that several of the formulae mentioned above have algebraic phenomena lying behind them. For example, the formula of Fa\`{a} di Bruno was shown to be related to a special type of Hopf algebra (see, e.g., \cite{[FGV]} and some of the references therein), and \cite{[S]} describes some of the algebra behind the Lagrange Inversion Formula. It is an interesting question whether an object of similar flavor can be related to our formulae.

This paper is divided into 4 sections. Section \ref{Blocks} defines the binomial combinations that are used throughout the paper. Section \ref{FirstDesc} determines the products of these combinations that show up in our main formula, while in Section \ref{DetCoeff} we determine the combinatorial coefficient with which every such expression appears. Finally, Section \ref{Expansion} explains where these combinations come from, as well as deduces the formula of \cite{[J3]} from ours.

I would like to thank W. P. Johnson for referring me to \cite{[N]}, for sharing the details of \cite{[J3]} with me, and for interesting discussions around this subject, as well as to T. Schlank for referring me to \cite{[STZ]}. Many thanks are due to the referee, for making suggestions for improving the manuscript, as well as for introducing me to the references \cite{[Wi]}, \cite{[C1]}, \cite{[C2]}, and \cite{[CF]}.

\section{The Basic Building Blocks \label{Blocks}}

Let $f$ be a function of two variables that is continuously differentiable with respect to both variables, and let $x_{0}$ and $y_{0}$ be such that $f(x_{0},y_{0})=0$ and $f_{y}(x_{0},y_{0})\neq0$. Here and throughout an index $x$ or $y$ of a function $g$ means its partial derivative with respect to that variable, so that the non-vanishing expression is $\frac{\partial f}{\partial y}(x_{0},y_{0})$. It is taught in every basic course in calculus that in this case the equation $f(x,y)=0$ determines $y$ as a differentiable function of $x$ in the neighborhood of $x_{0}$ (with $y(x_{0})=y_{0}$), and the derivative $y'(x_{0})$ is $-\frac{f_{x}(x_{0},y_{0})}{f_{y}(x_{0},y_{0})}$. From now on we shall omit the arguments $x_{0}$ and $y_{0}$, so that the latter equality is written more succinctly as $y'=-\frac{f_{x}}{f_{y}}$. Assuming that $f$ has all the derivatives of sufficiently high order, we aim to find a formula for the (higher) derivatives of $y$, namely $y^{(n)}=\frac{d^{n}y}{dx^{n}}$ with $n\geq2$.

Recall that when $g$ is a function of $x$ and $y$, the derivative of the function sending $x$ to $g\big(x,y(x)\big)$ is $g_{x}+g_{y}y'=g_{x}-\frac{g_{y}f_{x}}{f_{y}}$. In particular we get
\begin{equation}
\frac{d}{dx}f_{x}=\frac{f_{xx}f_{y}-f_{yx}f_{x}}{f_{y}}\qquad\mathrm{and}\qquad\frac{d}{dx}f_{y}=\frac{f_{xy}f_{y}-f_{yy}f_{x}}{f_{y}}, \label{ddx1stder}
\end{equation}
and we can evaluate $y''=-\frac{d}{d_{x}}\frac{f_{x}}{f_{y}}$, either using Equation \eqref{ddx1stder} or using the expansion $-\big(\frac{f_{x}}{f_{y}}\big)_{x}+\big(\frac{f_{x}}{f_{y}}\big)_{y}\frac{f_{x}}{f_{y}}$, as
\begin{equation}
\frac{-f_{xx}f_{y}+f_{yx}f_{x}}{f_{y}^{2}}+\frac{f_{xy}f_{y}-f_{yy}f_{x}}{f_{y}^{2}}\cdot\frac{f_{x}}{f_{y}}=\frac{-f_{xx}f_{y}^{2}+2f_{xy}f_{x}f_{y}-f_{yy}f_{x}^{2}}{f_{y}^{3}}. \label{yder2}
\end{equation}
The expressions $f_{xx}$, $f_{yy}$, and $f_{xy}=f_{yx}$ are, of course, the appropriate second derivatives of $f$, and derivatives of $f$ of higher order are denoted by additional indices. Note that we may use the symmetry of mixed derivatives since we always assume that $f$ is continuously differentiable enough times. We shall abbreviate the symbol $g_{x...xy...y}=\frac{\partial^{p+r}g}{\partial x^{p}\partial y^{r}}$ to simply $g_{x^{p}y^{r}}$ for any (smooth enough) function $g$ of $x$ and $y$. The undergraduate formula for $y'$ and Equation \eqref{yder2} for $y''$ also begin to identify the pattern, in which the denominator $y^{(n)}$ is $f_{y}^{2n-1}$ (note the similarity with the case of inverse functions, considered in \cite{[J2]}). We shall therefore be using the following lemma.
\begin{lem}
The expression $f_{y}^{2n+1}y^{(n+1)}$ can be evaluated as \[f_{y}^{2}\tfrac{d}{dx}\big(f_{y}^{2n-1}y^{(n)}\big)-(2n-1)(f_{xy}f_{y}-f_{yy}f_{x})\big(f_{y}^{2n-1}y^{(n)}\big).\] \label{denfree}
\end{lem}

\begin{proof}
Just take the derivative of $f_{y}^{2n-1}\big(x,y(x)\big)y^{(n)}$, apply the chain rule, multiply by $f_{y}^{2}$, evaluate $f_{y}\cdot\frac{d}{dx}f_{y}$ via Equation \eqref{ddx1stder}, and move the resulting expressions to the appropriate sides. This proves the lemma.
\end{proof}

Let us evaluate $f_{y}^{5}y'''$ via Lemma \ref{denfree}, where we recall from Equation \eqref{yder2} that the value of $f_{y}^{3}y''$ is $-f_{xx}f_{y}^{2}+2f_{xy}f_{x}f_{y}-f_{yy}f_{x}^{2}$. After expanding $\frac{d}{dx}f_{x}$ and $\frac{d}{dx}f_{y}$ as in Equation \eqref{ddx1stder}, the parts of $\frac{d}{dx}(f_{y}^{3}y'')$ arising from differentiating the first derivatives of $f$ (i.e., those involving $f_{xx}f_{xy}f_{y}^{2}$, $f_{xx}f_{yy}f_{x}f_{y}$, $f_{xy}^{2}f_{x}f_{y}$, and $f_{xy}f_{yy}f_{x}^{2}$) all cancel. Evaluating the remaining parts and adding the other term from Lemma \ref{denfree} shows that $f_{y}^{5}y'''$ is the sum of the expressions
\begin{equation}
-f_{xxx}f_{y}^{4}+3f_{xxy}f_{x}f_{y}^{3}-3f_{xyy}f_{x}^{2}f_{y}^{2}+f_{yyy}f_{x}^{3}f_{y} \label{yder3D3}
\end{equation}
and
\begin{equation}
3(f_{xy}f_{y}-f_{yy}f_{x})(f_{xx}f_{y}^{2}-2f_{xy}f_{x}f_{y}+f_{yy}f_{x}^{2}). \label{yder3rest}
\end{equation}
Equation \eqref{yder3rest} can be expanded, as in the explicit expressions appearing in \cite{[N]} and \cite{[J3]}, as $3f_{xx}f_{xy}f_{y}^{3}-3f_{xx}f_{yy}f_{x}f_{y}^{2}-6f_{xy}^{2}f_{x}f_{y}^{2}+9f_{xy}f_{yy}f_{x}^{2}f_{y}-3f_{yy}^{2}f_{x}^{3}$, but from our point of view the product in that equation will be more useful.

In any case, both our expression for $f_{y}^{3}y''$ and the first term in the formula for $f_{y}^{5}y'''$ involve some kind of combinatorial sum. We therefore make these combinatorial sums the ``basic building blocks'' for our formulae.
\begin{defn}
For a smooth enough function $g$ of the two variables $x$ and $y$ and an integer $l$ we define $\Delta_{l}g=\sum_{j=0}^{l}(-1)^{j}\binom{l}{j}g_{x^{l-j}y^{j}} \cdot f_{x}^{j}f_{y}^{l-j}$. \label{Deltalg}
\end{defn}
Using the expressions from Definition \ref{Deltalg} we find that the formula for $f_{y}^{3}y''$ in Equation \eqref{yder2} is just $-\Delta_{2}f$, and the expressions for the parts of $f_{y}^{5}y'''$ that appear in Equations \eqref{yder3D3} and \eqref{yder3rest} are $-f_{y}\Delta_{3}f$ and $+3\Delta_{1}f_{y}\cdot\Delta_{2}f$ respectively. Note that in our notation the index always precedes the $\Delta$-sign, i.e., the symbol $\Delta_{1}f_{y}$ should be understood as $\Delta_{1}(f_{y})$, and \emph{not} as $(\Delta_{1}f)_{y}$. One motivation for working with the symbols $\Delta_{l}g$ from Definition \ref{Deltalg} is that our formula for $f_{y}\cdot\frac{d}{dx}g\big(x,y(x)\big)$ yields just $\Delta_{1}g$, so that the expressions from Equation \eqref{ddx1stder} are $\Delta_{1}f_{x}$ and $\Delta_{1}f_{y}$ respectively, and we distinguish the latter from $(\Delta_{1}f)_{y}$ since an immediate calculation (resulting from the very definition of $y=y(x)$ as the value of $y$ for which $f\big(x,y(x)\big)=0$) shows that $\Delta_{1}f=0$.

\section{The Expressions Appearing in $f_{y}^{2n-1}y^{(n)}$ \label{FirstDesc}}

Since we work with the construction blocks from Definition \ref{Deltalg}, but going from $f_{y}^{2n-1}y^{(n)}$ to $f_{y}^{2n+1}y^{(n+1)}$ involves differentiation in Lemma \ref{denfree}, we have to evaluate the derivative of the expression from Definition \ref{Deltalg}. Recall again that the full derivative with respect to $x$ is of the expression $\Delta_{l}g\big(x,y(x)\big)$, where $y(x)$ is defined via $f(x,y)=0$ and therefore $y'=-\frac{f_{x}}{f_{y}}$.
\begin{lem}
$f_{y}^{2}\frac{d}{dx}\Delta_{l}g$ equals $f_{y}\Delta_{l+1}g+l\Delta_{1}f_{y}\cdot\Delta_{l}g-l\Delta_{2}f\cdot\Delta_{l-1}g_{y}$. \label{derDelta}
\end{lem}
Recall again that the relevant multiplier in the last term is $\Delta_{l-1}(g_{y})$ (and not $(\Delta_{l-1}g)_{y}$). Note that when $l=0$ the symbol $\Delta_{-1}$ is not required, because of the (vanishing) multiplier $l$.

\begin{proof}
Leibniz' rule shows that letting $f_{y}^{2}\frac{d}{dx}$ act on each term in Definition \ref{Deltalg} produces 3 elements: One in which the derivative operates on the derivative of $g$, one where it acts on the power of $f_{y}$, and one where it differentiates the power of $f_{x}$. The $j$th term of the first type becomes
\begin{equation}
\textstyle{(-1)^{j}\binom{l}{j}\big[g_{x^{l-j+1}y^{j}} \cdot f_{x}^{j}f_{y}^{l-j+2}-g_{x^{l-j}y^{j+1}} \cdot f_{x}^{j+1}f_{y}^{l-j+1}\big]}, \label{jderg}
\end{equation}
and if we replace the index $j$ by $j-1$ in the rightmost terms in Equation \eqref{jderg} the sum over $j$ becomes $\sum_{j=0}^{l+1}(-1)^{j}g_{x^{l-j+1}y^{j}}\big[\binom{l}{j}+\binom{l}{j-1}\big]f_{x}^{j}f_{y}^{l-j+2}$. As the latter combinatorial coefficient is just $\binom{l+1}{j}$ (by the classical binomial identity), this produces the first asserted term by Definition \ref{Deltalg}. In the remaining terms the summands with index $j$ combine to $(-1)^{j}\binom{l}{j}g_{x^{l-j}y^{j}}$ times
\begin{equation}
j(f_{xx}f_{x}^{j-1}f_{y}^{l-j+2}-f_{xy}f_{x}^{j}f_{y}^{l-j+1})+(l-j)(f_{yx}f_{x}^{j}f_{y}^{l-j+1}-f_{yy}f_{x}^{j+1}f_{y}^{l-j}). \label{jderf}
\end{equation}
The terms in Equation \eqref{jderf} that are multiplied by $l$ reduce to $\Delta_{1}f_{y} \cdot f_{x}^{j}f_{y}^{l-j}$, so that multiplying by the external coefficient and summing over $j$ produces the second required term by Definition \ref{Deltalg}. In the terms that are multiplied by $j$ in that equation we may of course assume that $j\geq1$, and after taking out a multiplying coefficient of $jf_{x}^{j-1}f_{y}^{l-j}$, one easily observes (using Definition \ref{Deltalg} again) that the remaining multipliers reduce to $\Delta_{2}f$. Putting in the external coefficient again, recalling that $j\binom{l}{j}=l\binom{l-1}{j-1}$, and taking the sign into account, we find that after replacing $j$ by $j+1$ these terms give $l\Delta_{2}f$ times $-\sum_{j=0}^{l-1}(-1)^{j}\binom{l-1}{j}g_{x^{l-1-j}y^{j+1}}f_{x}^{j}f_{y}^{l-1-j}$. As $g_{x^{l-1-j}y^{j+1}}$ equals $(g_{y})_{x^{l-1-j}y^{j}}$, this indeed produces the remaining desired term via Definition \ref{Deltalg}. This completes the proof of the lemma.
\end{proof}
We recall that the case $l=0$ in Lemma \ref{derDelta} is just the identity $\frac{d}{dx}g=g_{x}-\frac{g_{y}f_{x}}{f_{y}}$ from above, multiplied by $f_{y}^{2}$. We have also implicitly shown, in the proof of that lemma, that $\Delta_{2}f$ can be obtained as $f_{y}\Delta_{1}f_{x}-f_{x}\Delta_{1}f_{y}$. This is a special case of a more general formula, stating that $\Delta_{l+1}g=f_{y}\Delta_{l}g_{x}-f_{x}\Delta_{l}g_{y}$ for every $g$ and $l$. We can prove this identity by the usual binomial argument, but we shall not do so since we do not make use of this identity. Note that the cancellation that we used in proving the formula for $f_{y}^{5}y'''$ in Equations \eqref{yder3D3} and \eqref{yder3rest} amounts to the observation that when $g=f$ and $l=2$ the terms $\Delta_{1}f_{y}\cdot\Delta_{l}g$ and $\Delta_{2}f\cdot\Delta_{l-1}g_{y}$ from Lemma \ref{derDelta} coincide, and $f_{y}^{2}\frac{d}{dx}\Delta_{2}f$ is just $f_{y}\Delta_{3}f$. This special situation occurs in no other setting, and may hence be a bit misleading at first sight. In this particular case the terms $-f_{y}\Delta_{3}f$ and $+3\Delta_{1}f_{y}\cdot\Delta_{2}f$ of $f_{y}^{5}y'''$, appearing in Equations \eqref{yder3D3} and \eqref{yder3rest} respectively, correspond to the two terms in Lemma \ref{denfree} with $l=2$, because of the simple form of $f_{y}^{2}\frac{d}{dx}\Delta_{2}f$.

Let us apply Lemmas \ref{denfree} and \ref{derDelta} with the expression $-f_{y}\Delta_{3}f+3\Delta_{1}f_{y}\cdot\Delta_{2}f$ for $f_{y}^{5}y'''$ for evaluating $f_{y}^{7}y^{(4)}$. Applying Leibniz' rule for the differentiation in Lemma \ref{denfree}, we get from Lemma \ref{derDelta} that this part of $f_{y}^{7}y^{(4)}$ is the sum of $-f_{y}(f_{y}\Delta_{4}f+3\Delta_{1}f_{y}\cdot\Delta_{3}f-3\Delta_{2}f_{y}\cdot\Delta_{2}f)$, the term $-f_{y}\Delta_{1}f_{y}\cdot\Delta_{3}f$ arising from $g=f_{y}$ and $l=0$ in Lemma \ref{derDelta}, the term $+3\Delta_{1}f_{y} \cdot f_{y}\Delta_{3}f$ (using the cancellation in $f_{y}^{2}\frac{d}{dx}\Delta_{2}f$), and $+3\Delta_{2}f\big(f_{y}\Delta_{2}f_{y}+(\Delta_{1}f_{y})^{2}-f_{yy}\Delta_{2}f\big)$. Subtracting the other term from Lemma \ref{denfree} with $n=3$, namely $5\Delta_{1}f_{y}(-f_{y}\Delta_{3}f+3\Delta_{1}f_{y}\cdot\Delta_{2}f)$, and gathering similar terms, we find that $f_{y}^{7}y^{(4)}$ equals
\begin{equation}
-f_{y}^{2}\Delta_{4}f+4f_{y}\Delta_{1}f_{y}\cdot\Delta_{3}f+6f_{y}\Delta_{2}f_{y}\cdot\Delta_{2}f-3f_{yy}(\Delta_{2}f)^{2}-12(\Delta_{1}f_{y})^{2}\cdot\Delta_{2}f. \label{yder4}
\end{equation}

We can now characterize, up to numerical constants that we shall determine later, the terms appearing in $f_{y}^{2n-1}y^{(n)}$. Recall that for $n=2$ we only had a multiple of $\Delta_{2}f$, which is a single multiplier with no external derivative with respect to $y$ and an index 2 in the $\Delta$-symbol. For $n=3$ we had two products of two such symbols each, in each of which we had one $f$ and one $f_{y}$, and the sum of the $\Delta$-indices is 3. Considering the five terms in Equation \eqref{yder4}, we see that each of them is a product of three expressions, and in each product there is a total of 2 $y$-indices and the $\Delta$-indices sum to 4. We can thus state and prove the following result.
\begin{prop}
The expressions appearing in $f_{y}^{2n-1}y^{(n)}$ with $n\geq2$ are all products of $n-1$ terms of the form $\Delta_{l_{i}}f_{y^{r_{i}}}$, where in each such expression we have $\sum_{i=1}^{n-1}l_{i}=n$ and $\sum_{i=1}^{n-1}r_{i}=n-2$, and such that the terms $f$ (with $l=r=0$) and $\Delta_{1}f$ (in which $l=1$ and $r=0$) are not allowed. \label{formnoden}
\end{prop}

\begin{proof}
We have seen that the assertion holds for $n=2$ (as well as for some additional values of $n$), and arguing by induction we assume that it is true for $n$ and consider it for $n+1$, via Lemma \ref{denfree}. In the second term there, involving $\Delta_{1}f_{y} \cdot f_{y}^{2n-1}y^{(n)}$, we just add a single multiplier with $l=r=1$ (which is not one of the two excluded pairs), so that the resulting terms are indeed products of $n$ terms with the sum over the $l_{i}$'s (resp. the $r_{i}$'s) being $n+1$ (resp. $n-1$). For the first term we apply Leibniz' rule again, replacing a single term $\Delta_{l_{i}}f_{y^{r_{i}}}$ by the combination from Lemma \ref{derDelta}, and leaving the other terms invariant. The second term in that lemma again expresses simple multiplication by $\Delta_{1}f_{y}$, and we have treated this case already. In the first term we replace $l_{i}$ by $l_{i}+1$ and we have an extra term $f_{y}$ with $l=0$ and $r=1$, so that again we have one multiplier more and the sum of both the $l_{j}$'s and the $r_{j}$'s increase by 1. In the third term we replace $l_{i}$ by $l_{i}-1$ and $r_{i}$ by $r_{i}+1$, and multiply by $\Delta_{2}f$, with $l=2$ and $r=0$. As this operation also has the same effect on the number of multipliers and on the two sums, and as none of these operations introduce any of the disallowed expressions, this completes the proof of the proposition.
\end{proof}

Note that the induction step from the proof of Proposition \ref{formnoden} produces only products involving either $f_{y}$ (with $l=0$ and $r=1$), $\Delta_{1}f_{y}$ (in which $l=r=1$), or $\Delta_{2}f$ (having $l=2$ and $r=0$), but the assertion of that proposition does not mention that the products must involve one of these expressions. However, this is not an additional requirement, as we now see.
\begin{prop}
Let $\sqbinom{l_{i}}{r_{i}}$ be $n-1$ vectors with non-negative integral entries summing to $\sqbinom{n}{n-2}$ such that none of them equals $\sqbinom{0}{0}$ or $\sqbinom{1}{0}$. Then there exists some $i$ such that $\sqbinom{l_{i}}{r_{i}}$ is either $\sqbinom{0}{1}$, $\sqbinom{1}{1}$, or $\sqbinom{2}{0}$. \label{genrec}
\end{prop}

\begin{proof}
If $\sqbinom{l_{i}}{r_{i}}=\sqbinom{0}{1}$ for some $i$ then there is nothing to prove, so we assume that this is not the case. As $\sqbinom{0}{0}$ or $\sqbinom{1}{0}$ are also excluded, we find that $l_{i}+r_{i}\geq2$ for every $1 \leq i \leq n-1$. But then $\sum_{i=1}^{n-1}(l_{i}+r_{i})\geq\sum_{i=1}^{n-1}2=2n-2$, and as we have $\sum_{i=1}^{n-1}l_{i}=n$ and $\sum_{i=1}^{n-1}r_{i}=n-2$, we deduce that the equality $l_{i}+r_{i}=2$ must hold for every such $i$. But since the situation in which $\sqbinom{l_{i}}{r_{i}}=\sqbinom{0}{2}$ for every $i$ is impossible (since then the sum is $\sqbinom{0}{2n-2}$ and not $\sqbinom{n}{n-2}$), it follows that at least one of the vectors $\sqbinom{l_{i}}{r_{i}}$ must be one of the other two vectors with sum 2, namely $\sqbinom{1}{1}$ or $\sqbinom{2}{0}$. This proves the proposition.
\end{proof}
For $n=2$ the only possible vector is $\sqbinom{2}{0}$, and for $n=3$ the only allowed sums producing $\sqbinom{3}{1}$ are $\sqbinom{3}{0}+\sqbinom{0}{1}$ and $\sqbinom{2}{0}$+$\sqbinom{1}{1}$ (up to order). The unordered possible sums for $n=4$, giving $\sqbinom{4}{2}$, are easily verified to be precisely $\sqbinom{4}{0}+\sqbinom{0}{1}+\sqbinom{0}{1}$, $\sqbinom{3}{0}+\sqbinom{1}{1}+\sqbinom{0}{1}$, $\sqbinom{2}{1}+\sqbinom{2}{0}+\sqbinom{0}{1}$, $\sqbinom{2}{0}+\sqbinom{2}{0}+\sqbinom{0}{2}$, and $\sqbinom{2}{0}+\sqbinom{1}{1}+\sqbinom{1}{1}$. The expressions for the associated derivatives appearing in Equations \eqref{yder2}, \eqref{yder3D3}, \eqref{yder3rest}, and \eqref{yder4} show that in these cases all the associated terms indeed appear. At this stage we cannot yet say though, in spite of Proposition \ref{genrec}, that the terms in $f_{y}^{2n-1}y^{(n)}$ are in one-to-one correspondence with the products satisfying these conditions, since we do not yet know if the coefficient with which a specific product should appear in that expression vanishes (Theorem \ref{coeffval} below will show though that this is never the case).

It will later be more convenient to have an expression for the derivative $y^{(n)}$ alone, which is slightly modified. Observe that while $\sqbinom{0}{0}$ or $\sqbinom{1}{0}$ are excluded, the remaining vector with sum $<2$, namely $\sqbinom{0}{1}$, is associated with the expression $f_{y}$, which in $y^{(n)}$ itself appears in the denominator. The expressions from Propositions \ref{formnoden} and \ref{genrec} represent certain partitions of the vector $\sqbinom{n}{n-2}$, and we have seen in the Introduction that such partitions can be described using multiplicities. We therefore make the following definition for the set arising from Propositions \ref{formnoden} and \ref{genrec} and for another set that will be used below.
\begin{defn}
Denote by $\tilde{A}_{n}$ the set of partitions of $\sqbinom{n}{n-2}$ into $n-1$ (unordered) non-negative integral vectors which cannot be $\sqbinom{1}{0}$ or $\sqbinom{0}{0}$. We consider elements $\tilde{\alpha}\in\tilde{A}_{n}$ as multiplicities $m_{l,r}\geq0$ for non-negative integral $l$ and $r$, for which we have $m_{0,0}=m_{1,0}=0$ and the equalities $\sum_{l,r}m_{l,r}=n-1$, $\sum_{l,r}lm_{l,r}=n$ and $\sum_{l,r}rm_{l,r}=n-2$. In addition, let $A_{n}$ be the set of multiplicities $\{m_{l,r}\}_{\{l+r\geq2\}}$ for which the equalities $\sum_{l,r}lm_{l,r}=n$ and $\sum_{l,r}(r-1)m_{l,r}=-1$ hold. \label{Andef}
\end{defn}

\begin{lem}
Given $n\geq2$, in both the sets $\tilde{A}_{n}$ and $A_{n}$ from Definition \ref{Andef} only finitely many $m_{l,r}$'s can be non-zero. These sets are canonically isomorphic. \label{AntildeAn}
\end{lem}

\begin{proof}
Given an element of $A_{n}$, we have $\sum_{l+r\geq2}m_{l,r}\leq\sum_{l+r\geq2}(l+r-1)m_{l,r}$ (since in each summand we have $l+r-1\geq1$ and $m_{l,r}\geq0$), and the latter sum equals $n-1$ by Definition \ref{Andef}. As the corresponding sum equals $n-1$ for elements of $\tilde{A}_{n}$ as well, the finiteness of the non-zero $m_{l,r}$'s is clear for both sets. Consider now the map sending $\tilde{\alpha}=\{m_{l,r}\}_{l,r}\in\tilde{A}_{n}$ to $\alpha=\{m_{l,r}\}_{\{l+r\geq2\}}$. Since in $\tilde{\alpha}$ we have $m_{0,0}=m_{1,0}=0$, it amounts to omitting the multiplicity $m_{0,1}$. It is clear from Definition \ref{Andef} that elements of $\tilde{A}_{n}$ also satisfy the equality $\sum_{l,r}(r-1)m_{l,r}=-1$. As for $m_{0,1}$ both $l$ and $r-1$ vanish, omitting this multiplicity affects neither $\sum_{l,r}lm_{l,r}$ nor $\sum_{l,r}(r-1)m_{l,r}$, implying that our element $\alpha$ indeed lies in $A_{n}$. In addition, the value of $m_{0,1}$ is determined by the others via the equality $n-1=\sum_{l,r}m_{l,r}=\sum_{l+r\geq2}m_{l,r}+m_{0,1}$, so that that the image of an element $\tilde{\alpha}$ in $A_{n}$ determines $m_{0,1}$ and our map is also injective. Moreover, the inequality established in the beginning of the proof implies that this determined value of $m_{0,1}$ is non-negative for every $\alpha \in A_{n}$. Therefore our canonical map $\tilde{\alpha}\mapsto\alpha$ is a bijection, which proves the lemma.
\end{proof}

We therefore obtain the following description of $y^{(n)}$.
\begin{cor}
For every $\alpha \in A_{n}$ there is a coefficient $c_{\alpha}$ such that $y^{(n)}$ can be written as $\sum_{\alpha=\{m_{l,r}\}_{\{l+r\geq2\}} \in A_{n}}\Big[c_{\alpha}\prod_{l+r\geq2}(\Delta_{l}f_{y^{r}})^{m_{l,r}}\Big/f_{y}^{n+\sum_{l+r\geq2}m_{l,r}}\Big]$. \label{formwden}
\end{cor}

\begin{proof}
Propositions \ref{formnoden} and \ref{genrec} and Definition \ref{Andef} allow us to present $y^{(n)}$ as the sum $\sum_{\tilde{\alpha}=\{m_{l,r}\}_{l,r}\in\tilde{A}_{n}}c_{\tilde{\alpha}}\prod_{l,r}(\Delta_{l}f_{y^{r}})^{m_{l,r}}\big/f_{y}^{2n-1}$. We now identify each element $\tilde{\alpha}\in\tilde{A}_{n}$ with its image $\alpha \in A_{n}$ via Lemma \ref{AntildeAn} and write $(\Delta_{0}f_{y^{1}})^{m_{0,1}}/f_{y}^{2n-1}$ as $f_{y}^{m_{0,1}-2n+1}$ (as well as $c_{\alpha}=c_{\tilde{\alpha}}$), so that substituting the value of $m_{0,1}$ from the proof of that lemma gives the desired result. This proves the corollary.
\end{proof}

\begin{rmk}
Observe that the sum $h=\sum_{l+r\geq2}m_{l,r}$ is just the number of vectors appearing in $\alpha \in A_{n}$, which is now a partition of $\sqbinom{n}{h-1}$ into $h$ vectors. The proof of Lemma \ref{AntildeAn} shows that
this number $h$ satisfies $1 \leq h \leq n-1$: Indeed, the fact that $h\geq1$ is immediate (since none of the equalities from Definition \ref{Andef} can be satisfied when $m_{l,r}=0$ for every $l+r\geq2$), and as the value $n+h-1$ of $\sum_{l+r\geq2}(l+r)m_{l,r}$ must be at least $\sum_{l+r\geq2}2m_{l,r}=2h$, one also deduces the other inequality. It follows that we can present $A_{n}$ as the disjoint union $\bigcup_{h=1}^{n-1}A_{n,h}$, where $A_{n,h}$ consists of those elements of $A_{n}$ for which $\sum_{l+r\geq2}m_{l,r}=h$ (or equivalently $\sum_{l+r\geq2}rm_{l,r}=h-1$), and the denominator under the terms arising from $\alpha \in A_{n,h}$ in Corollary \ref{formwden} is $f_{y}^{n+h}$. On the other hand, if $A$ is the set of elements $\alpha=\{m_{l,r}\}_{\{l+r\geq2\}}$ satisfying only $m_{l,r}\geq0$, the equality $\sum_{l,r}(r-1)m_{l,r}=-1$, and the finiteness condition from Lemma \ref{AntildeAn}, then we claim that $A$ is the disjoint union $\bigcup_{n=2}^{\infty}A_{n}$. Indeed, the index $n$ such that $\alpha \in A_{n}$ is $\sum_{l+r\geq2}lm_{l,r}$, and we need to show that it must be at least 2. But adding $-1=\sum_{l,r}(r-1)m_{l,r}$ to the latter sum gives $n-1=\sum_{l+r\geq2}(l+r-1)m_{l,r}$, which is at least $\sum_{l+r\geq2}m_{l,r}=h\geq1$, so that the union defining $A$ indeed begins with $n=2$. \label{setpart}
\end{rmk}

\section{The Combinatorial Coefficients \label{DetCoeff}}

We now turn to evaluating the coefficients $c_{\alpha}$ from Corollary \ref{formwden}. It is $-1$ for the single element of $A_{2}$ if $n=2$, when $n=3$ the two elements of $A_{3}$ come with the coefficients $-1$ and $+3$, and for $n=4$ they are easily read from Equation \eqref{yder4} (after dividing by $f_{y}^{7}$). Before we give an explicit formula for them, we shall require their behavior under the inductive definition arising from differentiating and Lemma \ref{denfree}. For this we take an element $\tilde{\alpha}$ of the set $\tilde{A}_{n}$ from Corollary \ref{formwden}, with the multiplicities $\{m_{l,r}\}_{l,r}$ as in that corollary, and we introduce the following notation. Given a vector $\sqbinom{\tilde{l}}{\tilde{r}}$ with $\tilde{l}+\tilde{r}\geq2$ and $m_{\tilde{l},\tilde{r}}\geq1$ we define, for every $l$ and $r$, the multiplicity $m_{+,l,r}^{\tilde{l},\tilde{r}}$ to be $m_{l,r}+1$ if $\sqbinom{l}{r}$ is $\sqbinom{0}{1}$ or $\sqbinom{\tilde{l}+1}{\tilde{r}}$, $m_{l,r}-1$ in case $\sqbinom{l}{r}=\sqbinom{\tilde{l}}{\tilde{r}}$, and just $m_{l,r}$ otherwise. We denote $\{m_{+,l,r}^{\tilde{l},\tilde{r}}\}_{l,r}$ by $\tilde{\alpha}_{+}^{\tilde{l},\tilde{r}}$. In case $\tilde{l}\geq1$, $m_{\tilde{l},\tilde{r}}\geq1$, and $\sqbinom{\tilde{l}}{\tilde{r}}\neq\sqbinom{2}{0}$ we also set $m_{t,l,r}^{\tilde{l},\tilde{r}}$ to be $m_{l,r}+1$ in case $\sqbinom{l}{r}$ is either $\sqbinom{2}{0}$ or $\sqbinom{\tilde{l}-1}{\tilde{r}+1}$, $m_{l,r}-1$ when $\sqbinom{l}{r}=\sqbinom{\tilde{l}}{\tilde{r}}$, and $m_{l,r}$ in any other case, and let $\tilde{\alpha}_{t}^{\tilde{l},\tilde{r}}=\{m_{t,l,r}^{\tilde{l},\tilde{r}}\}_{l,r}$. Finally, set $m_{m,l,r}$ to be $m_{l,r}+1$ if $\sqbinom{l}{r}=\sqbinom{1}{1}$ and $m_{l,r}$ when $\sqbinom{l}{r}\neq\sqbinom{1}{1}$, and denote $\{m_{m,l,r}\}_{l,r}$ by $\tilde{\alpha}_{m}$. We can now prove the following result.
\begin{lem}
For $\tilde{\alpha}\in\tilde{A}_{n}$ and appropriate $\tilde{l}$ and $\tilde{r}$ the elements $\tilde{\alpha}_{+}^{\tilde{l},\tilde{r}}$, $\tilde{\alpha}_{t}^{\tilde{l},\tilde{r}}$, and $\tilde{\alpha}_{m}$ all lie in $\tilde{A}_{n+1}$. Moreover, the image of $c_{\tilde{\alpha}}\prod_{l,r}(\Delta_{l}f_{y^{r}})^{m_{l,r}}$ under the operation $f_{y}^{2}\frac{d}{dx}-(2n-1)\Delta_{1}f_{y}$ consists of the following terms: For every $\tilde{l}$ and $\tilde{r}$ with $\tilde{l}+\tilde{r}\geq2$ the product associated with $\tilde{\alpha}_{+}^{\tilde{l},\tilde{r}}$ comes with the coefficient $m_{\tilde{l},\tilde{r}}c_{\tilde{\alpha}}$; If $\tilde{l}\geq1$ and $\sqbinom{\tilde{l}}{\tilde{r}}\neq\sqbinom{2}{0}$ then the product corresponding to $\tilde{\alpha}_{t}^{\tilde{l},\tilde{r}}$ appears with the coefficient $-\tilde{l}m_{\tilde{l},\tilde{r}}c_{\tilde{\alpha}}$; And the only remaining term is $-(n-1-m_{0,1}+2m_{2,0})c_{\tilde{\alpha}}$ times the product arising from $\tilde{\alpha}_{m}$. \label{contcalpha}
\end{lem}

\begin{proof}
We apply Leibniz' rule to $f_{y}^{2}\frac{d}{dx}\big[c_{\tilde{\alpha}}\prod_{l,r}(\Delta_{l}f_{y^{r}})^{m_{l,r}}\big]$, and in the summand associated with $\tilde{l}$ and $\tilde{r}$ (in which we have a coefficient of $m_{\tilde{l},\tilde{r}}$ from the exponent) we expand $f_{y}^{2}\frac{d}{dx}\Delta_{\tilde{l}}f_{y^{\tilde{r}}}$ as in Lemma \ref{derDelta}. When $\tilde{l}+\tilde{r}\geq2$ the first term from that lemma produces the required term associated with $\tilde{\alpha}_{+}^{\tilde{l},\tilde{r}}$, since the multiplicity $m_{0,1}$ of $f_{y}$ increases by 1, and one vector $\sqbinom{\tilde{l}}{\tilde{r}}$ is replaced by $\sqbinom{\tilde{l}+1}{\tilde{r}}$. If $\tilde{l}\geq1$ (which implies $\tilde{l}+\tilde{r}\geq2$ again since $\sqbinom{0}{1}$ is excluded) and $\sqbinom{\tilde{l}}{\tilde{r}}\neq\sqbinom{2}{0}$, then the third term in that lemma gives the desired term corresponding to $\tilde{\alpha}_{t}^{\tilde{l},\tilde{r}}$, since $\Delta_{2}f$ increases $m_{2,0}$ by 1 and one copy of $\sqbinom{\tilde{l}}{\tilde{r}}$ becomes $\sqbinom{\tilde{l}-1}{\tilde{r}+1}$.

The remaining expressions are the second term from Lemma \ref{denfree}, all the second terms from Lemma \ref{derDelta}, the first term from Lemma \ref{derDelta} with $\sqbinom{\tilde{l}}{\tilde{r}}=\sqbinom{0}{1}$, and the third term from that lemma when $\sqbinom{\tilde{l}}{\tilde{r}}=\sqbinom{2}{0}$. It is rather evident that all of these expressions are multiples of the product associated with $\tilde{\alpha}_{m}$, with the respective coefficients $-(2n-1)c_{\tilde{\alpha}}$, $+\sum_{l,r}lm_{l,r}c_{\tilde{\alpha}}$ (which becomes just $+nc_{\tilde{\alpha}}$ by Definition \ref{Andef} since $\tilde{\alpha}\in\tilde{A}_{n}$), $+m_{0,1}c_{\tilde{\alpha}}$, and $-2m_{2,0}c_{\tilde{\alpha}}$, and they sum to the asserted total contribution there. The fact that all the elements in question belong to $\tilde{A}_{n+1}$ either follows form our argument combined with Proposition \ref{formnoden}, or can be easily seen directly. This proves the lemma.
\end{proof}

We shall also need the notation that is dual to the one appearing in Lemma \ref{contcalpha}, and we shall introduce it for an element $\beta=\{\mu_{l,r}\}_{\{l+r\geq2\}}$ of $A_{n+1}$ (and not of $\tilde{A}_{n+1}$). For such an element $\beta$ we define, in case the vector $\sqbinom{\hat{l}}{\hat{r}}$ satisfies $\hat{l}\geq1$, $\hat{l}+\hat{r}\geq3$, and $\mu_{\hat{l},\hat{r}}\geq1$, the multiplicity $\mu_{-,l,r}^{\hat{l},\hat{r}}$ to be $\mu_{l,r}-1$ when $\sqbinom{l}{r}=\sqbinom{\hat{l}}{\hat{r}}$, $\mu_{l,r}+1$ if $\sqbinom{l}{r}=\sqbinom{\hat{l}-1}{\hat{r}}$, and $\mu_{l,r}$ in any other case. Denote the element $\{\mu_{-,l,r}^{\hat{l},\hat{r}}\}_{\{l+r\geq2\}}$ by $\beta_{-}^{\hat{l},\hat{r}}$. Assuming that $\mu_{2,0}\geq1$ and $\sqbinom{\hat{l}}{\hat{r}}$ satisfies $\hat{r}\geq1$, $\sqbinom{\hat{l}}{\hat{r}}\neq\sqbinom{1}{1}$, and $\mu_{\hat{l},\hat{r}}\geq1$ once again, we set $\mu_{b,l,r}^{\hat{l},\hat{r}}$ to be $\mu_{l,r}-1$ if $\sqbinom{l}{r}$ is either $\sqbinom{\hat{l}}{\hat{r}}$ or $\sqbinom{2}{0}$, $\mu_{l,r}+1$ in case $\sqbinom{l}{r}=\sqbinom{\hat{l}+1}{\hat{r}-1}$, and $\mu_{l,r}$ otherwise, and then $\{\mu_{b,l,r}^{\hat{l},\hat{r}}\}_{\{l+r\geq2\}}$ is denoted by $\beta_{b}^{\hat{l},\hat{r}}$. Finally, in case $\mu_{1,1}\geq1$ we let $\mu_{d,l,r}$ be $\mu_{l,r}-1$ in case $\sqbinom{l}{r}=\sqbinom{1}{1}$ and $\mu_{l,r}$ if $\sqbinom{l}{r}\neq\sqbinom{1}{1}$, and set $\beta_{d}=\{\mu_{d,l,r}\}_{\{l+r\geq2\}}$. The duality of the two notions is proved as follows, where putting and omitting the tilde means applying the canonical map from Lemma \ref{AntildeAn}, with the respective index and in the appropriate direction.
\begin{lem}
For $n\geq2$ and $\beta \in A_{n+1}$ all the expressions $\beta_{-}^{\hat{l},\hat{r}}$, $\beta_{b}^{\hat{l},\hat{r}}$, and $\beta_{d}$ lie in $A_{n}$ (when they are defined). Let $\tilde{\alpha}\in\tilde{A}_{n}$ as well as two vectors $\sqbinom{\tilde{l}}{\tilde{r}}$ and $\sqbinom{\hat{l}}{\hat{r}}$ be also given. Assume that $\sqbinom{\hat{l}}{\hat{r}}=\sqbinom{\tilde{l}+1}{\tilde{r}}$ (or equivalently $\sqbinom{\tilde{l}}{\tilde{r}}=\sqbinom{\hat{l}-1}{\hat{r}}$), and then $\tilde{\alpha}_{+}^{\tilde{l},\tilde{r}}$ is defined and equals $\tilde{\beta}$ if and only if $\beta_{-}^{\hat{l},\hat{r}}$ is defined and equals $\alpha$. On the other hand, if $\sqbinom{\hat{l}}{\hat{r}}=\sqbinom{\tilde{l}-1}{\tilde{r}+1}$ (which is equivalent to $\sqbinom{\tilde{l}}{\tilde{r}}=\sqbinom{\hat{l}+1}{\hat{r}-1}$) then $\tilde{\alpha}_{t}^{\tilde{l},\tilde{r}}$ is defined and equals $\tilde{\beta}$ precisely when $\beta_{b}^{\hat{l},\hat{r}}$ is defined and equals $\alpha$. Finally, $\tilde{\beta}$ is of the form $\tilde{\alpha}_{m}$ if and only if $\beta_{d}$ is defined and equals $\alpha$. \label{dualnot}
\end{lem}

\begin{proof}
The two conditions on $\beta_{-}^{\hat{l},\hat{r}}$, $\beta_{b}^{\hat{l},\hat{r}}$, and $\beta_{d}$ from Definition \ref{Andef} are easily verified using the fact that $\beta \in A_{n+1}$, proving the first assertion. For the second one the condition $\tilde{l}+\tilde{r}\geq2$ implies $\hat{l}\geq1$ and $\hat{l}+\hat{r}\geq3$ for the value of $\sqbinom{\hat{l}}{\hat{r}}$, and on the other hand from $\hat{l}\geq1$ and $\hat{l}+\hat{r}\geq3$ we deduce $\tilde{l}+\tilde{r}\geq2$ with our $\sqbinom{\tilde{l}}{\tilde{r}}$. In addition, with this relation between $\sqbinom{\hat{l}}{\hat{r}}$ and $\sqbinom{\tilde{l}}{\tilde{r}}$ the two conditions $\tilde{\beta}=\tilde{\alpha}_{+}^{\tilde{l},\tilde{r}}$ (namely $\mu_{l,r}=m_{+,l,r}^{\tilde{l},\tilde{r}}$ for every $l$ and $r$) and $\alpha=\beta_{-}^{\hat{l},\hat{r}}$ (i.e., $m_{l,r}=\mu_{-,l,r}^{\hat{l},\hat{r}}$ wherever $l+r\geq2$) both mean that for $l+r\geq2$ one has $\mu_{l,r}=m_{l,r}$ when $\sqbinom{l}{r}$ equals neither $\sqbinom{\hat{l}}{\hat{r}}$ nor $\sqbinom{\tilde{l}}{\tilde{r}}$, while when $\sqbinom{l}{r}=\sqbinom{\hat{l}}{\hat{r}}$ we have the equivalent equalities $\mu_{\hat{l},\hat{r}}=m_{\hat{l},\hat{r}}+1\geq1$ and $m_{\hat{l},\hat{r}}=\mu_{\hat{l},\hat{r}}-1$ and with $\sqbinom{l}{r}=\sqbinom{\tilde{l}}{\tilde{r}}$ we get $\mu_{\tilde{l},\tilde{r}}=m_{\tilde{l},\tilde{r}}-1$ and $m_{\tilde{l},\tilde{r}}=\mu_{\tilde{l},\tilde{r}}+1\geq1$ (the inequalities covering more admissibility conditions). The remaining equality $\mu_{0,1}=m_{0,1}+1$ is now a consequence is the fact that $\tilde{\beta}\in\tilde{A}_{n+1}$ and $\tilde{\alpha}_{+}^{\tilde{l},\tilde{r}}\in\tilde{A}_{n}$, and the proof of Proposition \ref{genrec} shows that if $\mu_{\hat{l},\hat{r}}\geq1$ for some $\sqbinom{\hat{l}}{\hat{r}}$ with $\hat{l}+\hat{r}\geq3$ then $\mu_{0,1}\geq1$. The second assertion is thus established.

As for the third one, we note that with the current relation between $\sqbinom{\hat{l}}{\hat{r}}$ and $\sqbinom{\tilde{l}}{\tilde{r}}$ we have $\hat{r}\geq1$ and $\tilde{l}\geq1$, the common sum $\tilde{l}+\tilde{r}=\hat{l}+\hat{r}$ is at least 2, and $\sqbinom{\tilde{l}}{\tilde{r}}\neq\sqbinom{2}{0}$ precisely when $\sqbinom{\hat{l}}{\hat{r}}\neq\sqbinom{1}{1}$. Then the condition $\tilde{\beta}=\tilde{\alpha}_{t}^{\tilde{l},\tilde{r}}$ (which means $\mu_{l,r}=m_{t,l,r}^{\tilde{l},\tilde{r}}$ for every $l$ and $r$) and the condition $\alpha=\beta_{b}^{\hat{l},\hat{r}}$ (namely $m_{l,r}=\mu_{b,l,r}^{\hat{l},\hat{r}}$ wherever $l+r\geq2$) both amount, for any vector $\sqbinom{l}{r}$ with $l+r\geq2$, to the following equalities: $\mu_{l,r}=m_{l,r}$ if $\sqbinom{l}{r}$ is not any of the vectors $\sqbinom{\hat{l}}{\hat{r}}$, $\sqbinom{\tilde{l}+1}{\tilde{r}}$, or $\sqbinom{2}{0}$; $\mu_{\hat{l},\hat{r}}=m_{\hat{l},\hat{r}}+1\geq1$ and $m_{\hat{l},\hat{r}}=\mu_{\hat{l},\hat{r}}-1$ for $\sqbinom{l}{r}=\sqbinom{\hat{l}}{\hat{r}}$; $\mu_{\tilde{l},\tilde{r}}=m_{\tilde{l},\tilde{r}}-1$ and $m_{\tilde{l},\tilde{r}}=\mu_{\tilde{l},\tilde{r}}+1\geq1$ when $\sqbinom{l}{r}=\sqbinom{\tilde{l}}{\tilde{r}}$; and $\mu_{2,0}=m_{2,0}+1\geq1$ and $m_{2,0}=\mu_{2,0}-1$ in case $\sqbinom{l}{r}=\sqbinom{2}{0}$. Since the inequalities provide the remaining admissibility conditions, and the fact that $\tilde{\beta}\in\tilde{A}_{n+1}$ and $\tilde{\alpha}_{t}^{\tilde{l},\tilde{r}}\in\tilde{A}_{n}$ imply the last inequality $\mu_{0,1}=m_{0,1}$, this proves the third assertion.

The fourth assertion is simpler: From both $\tilde{\beta}=\tilde{\alpha}_{m}$ and $\alpha=\beta_{d}$ we get $\mu_{l,r}=m_{l,r}$ wherever $l+r\geq2$ and $\sqbinom{l}{r}\neq\sqbinom{1}{1}$ as well as $\mu_{1,1}=m_{1,1}+1\geq1$ (including the admissibility condition) and $m_{1,1}=\mu_{1,1}-1$. Then from $\tilde{\beta}\in\tilde{A}_{n+1}$ and $\tilde{\alpha}_{m}\in\tilde{A}_{n}$ we also obtain $\mu_{0,1}=m_{0,1}$. This proves the lemma.
\end{proof}

For writing the recursive formula for the coefficients $c_{\alpha}$ from Corollary \ref{formwden} we shall need the Kronecker $\delta$-symbol $\delta_{i,j}$, which equals 1 in case $i=j$ and 0 otherwise. More precisely, we shall use its complement $\overline{\delta}_{i,j}=1-\delta_{i,j}$. Writing $y^{(n)}$ and $y^{(n+1)}$ as in Corollary \ref{formwden}, we can express $c_{\beta}$ as follows.
\begin{cor}
Given $n\geq2$ and $\beta \in A_{n+1}$, the coefficient $c_{\beta}$ equals \[\sum_{\hat{l}+\hat{r}\geq2}\!\overline{\delta}_{\mu_{\hat{l},\hat{r}},0}\!\Big[\overline{\delta}_{\hat{l},0}\overline{\delta}_{\hat{l}+\hat{r},2}(\mu_{\hat{l}-1,\hat{r}}+1)c_{\beta_{-}^{\hat{l},\hat{r}}}-
\overline{\delta}_{\mu_{2,0},0}\overline{\delta}_{\hat{r},0}\overline{\delta}_{\sqbinom{\hat{l}}{\hat{r}},\sqbinom{1}{1}}(\hat{l}+1)(\mu_{\hat{l}+1,\hat{r}-1}+1)c_{\beta_{b}^{\hat{l},\hat{r}}}\Big]\] minus the expression $\overline{\delta}_{\mu_{1,1},0}\big(\sum_{l+r\geq2}r\mu_{l,r}+2\mu_{2,0}\big)c_{\beta_{d}}$. \label{contcbeta}
\end{cor}

\begin{proof}
Expressing $f_{y}^{2n-1}y^{(n)}$ and $f_{y}^{2n+1}y^{(n+1)}$ using the multiplicities (as in the proof of Corollary \ref{formwden}), and recalling from Lemma \ref{denfree} that the latter is the image of the former under $f_{y}^{2}\frac{d}{dx}-(2n-1)\Delta_{1}f_{y}$, we need to gather the contributions to the product associated with the element $\tilde{\beta}\in\tilde{A}_{n+1}$ corresponding to $\beta$. By Lemma \ref{contcalpha}, such contributions occur precisely from those $\tilde{\alpha}\in\tilde{A}_{n}$ for which $\tilde{\beta}$ is $\tilde{\alpha}_{+}^{\tilde{l},\tilde{r}}$ or $\tilde{\alpha}_{t}^{\tilde{l},\tilde{r}}$ for some admissible vector $\sqbinom{\tilde{l}}{\tilde{r}}$ or for which $\tilde{\beta}=\tilde{\alpha}_{m}$, and we denote, for every element $\tilde{\alpha}\in\tilde{A}_{n}$, the corresponding element of $A_{n}$ by $\alpha$ as before. Lemma \ref{dualnot} shows that $\tilde{\beta}$ is $\tilde{\alpha}_{+}^{\hat{l}-1,\hat{r}}$ for $\alpha=\beta_{-}^{\hat{l},\hat{r}}$ wherever $\hat{l}\geq1$, $\hat{l}+\hat{r}\geq3$, and $\mu_{\hat{l},\hat{r}}\geq1$, it is $\tilde{\alpha}_{t}^{\hat{l}+1,\hat{r}-1}$ with $\alpha=\beta_{b}^{\hat{l},\hat{r}}$ when $\hat{r}\geq1$, $\sqbinom{\hat{l}}{\hat{r}}\neq\sqbinom{1}{1}$, and $\mu_{\hat{l},\hat{r}}\geq1$ in case $\mu_{2,0}\geq1$, and if $\mu_{1,1}\geq1$ then it is also $\tilde{\alpha}_{m}$ where $\alpha$ is $\beta_{d}$. It follows that $c_{\beta}$ is the sum of the resulting contributions, and substituting the values of the parameters $\tilde{l}$ and $m_{\tilde{l},\tilde{r}}$ associated with each vector $\sqbinom{\hat{l}}{\hat{r}}$ inside the relevant contributions from Lemma \ref{contcalpha} immediately gives the asserted sum over $\hat{l}+\hat{r}\geq2$ (the $\overline{\delta}$ expressions are there to enforce the admissibility conditions). As for the multiplier of $c_{\beta_{d}}$ in case $\mu_{1,1}\geq1$, recall that the value of $m_{0,1}$ was seen to be $n-1-\sum_{l+r\geq2}m_{l,r}$ in the proof of Lemma \ref{AntildeAn}, and the sum here also equals $\sum_{l+r\geq2}rm_{l,r}+1$ by Definition \ref{Andef}. By observing that $m_{l,r}$ is $\mu_{l,r}-1$ when $l=r=1$ and $\mu_{l,r}$ otherwise (when $\alpha=\beta_{d}$), we indeed obtain the asserted value. This completes the proof of the corollary.
\end{proof}

We can now define the combinatorial numbers that we shall soon need.
\begin{defn}
For $\alpha=\{m_{l,r}\}_{\{l+r\geq2\}}$ in the set $A$ from Remark \ref{setpart} we set \[C_{\alpha}=\Bigg(\sum_{l+r\geq2}lm_{l,r}\Bigg)!\Bigg(\sum_{l+r\geq2}rm_{l,r}\Bigg)!\Bigg/\prod_{l+r\geq2}l!^{m_{l,r}}r!^{m_{l,r}}m_{l,r}!.\] When $\alpha$ is in $A_{n,h}$, this is the number of possibilities to put $n$ marked blue balls and $h-1$ marked red balls into $h$ identical boxes such that for every $l$ and $r$ (with $l+r\geq2$) there are $m_{l,r}$ boxes containing $l$ blue balls and $r$ red balls. \label{combcoeff}
\end{defn}
Note that the expressions $m_{l,r}!$ appear in the denominator of $C_{\alpha}$ in Definition \ref{combcoeff} since the boxes in the combinatorial description there are identical (this is also the case with the coefficients in Fa\`{a} di Bruno's formula---see \cite{[J1]}). The technical property of the coefficients from Definition \ref{combcoeff} that we shall require is the following one.
\begin{prop}
For $n\geq2$ and $\beta=\{\mu_{l,r}\}_{\{l+r\geq2\}} \in A_{n+1}$ the number $C_{\beta}$ is \[\sum_{\hat{l}+\hat{r}\geq2}\!\!\overline{\delta}_{\mu_{\hat{l},\hat{r}},0}\!\Big[\overline{\delta}_{\hat{l},0}\overline{\delta}_{\hat{l}+\hat{r},2}(\mu_{\hat{l}-1,\hat{r}}+1)C_{\beta_{-}^{\hat{l},\hat{r}}}+
\overline{\delta}_{\mu_{2,0},0}\overline{\delta}_{\hat{r},0}\overline{\delta}_{\sqbinom{\hat{l}}{\hat{r}}\!,\sqbinom{1}{1}}\!(\hat{l}+1)(\mu_{\hat{l}+1,\hat{r}-1}+1)C_{\beta_{b}^{\hat{l},\hat{r}}}\!\Big]\] plus the two terms $2\overline{\delta}_{\mu_{1,1},0}\mu_{2,0}C_{\beta_{d}}$ and $\overline{\delta}_{\mu_{1,1},0}\sum_{l+r\geq2}r\mu_{l,r} \cdot C_{\beta_{d}}$.
\label{indcomb}
\end{prop}
The separation of the two terms involving $\overline{\delta}_{\mu_{1,1},0}$ will be more convenient for the proof, as well as for the combinatorial explanation below.

\begin{proof}
Denote $\sum_{l+r\geq2}\mu_{l,r}$ by $h$, so that $\beta \in A_{n+1,h}$, and substitute the given values of $C_{\alpha}$ with $\alpha \in A_{n}$ into the asserted formula. For each $\hat{l}\geq1$ and $\hat{r}$ such that $\hat{l}+\hat{r}\geq3$ and $\mu_{\hat{l},\hat{r}}\geq1$ (these admissibility conditions are expressed in the $\overline{\delta}$-symbols) we take $\mu_{\hat{l}-1,\hat{r}}+1$ times $n!(h-1)!\big/\prod_{l+r\geq2}\mu_{-,l,r}^{\hat{l},\hat{r}}!(l!r!)^{\mu_{-,l,r}^{\hat{l},\hat{r}}}$. We substitute the values of $\mu_{-,l,r}^{\hat{l},\hat{r}}$, which also imply that $\beta_{-}^{\hat{l},\hat{r}} \in A_{n,h}$. For every $\sqbinom{l}{r}$ that is not $\sqbinom{\hat{l}}{\hat{r}}$ or $\sqbinom{\hat{l}-1}{\hat{r}}$ we get the same denominator as in $C_{\beta}$, and the remaining powers $\mu_{-,\hat{l},\hat{r}}^{\hat{l},\hat{r}}=\mu_{\hat{l},\hat{r}}-1$ and $\mu_{-,\hat{l}-1,\hat{r}}^{\hat{l},\hat{r}}=\mu_{\hat{l}-1,\hat{r}}+1$ of $\hat{r}!$ sum to $\mu_{\hat{l},\hat{r}}+\mu_{\hat{l}-1,\hat{r}}$ as in $C_{\beta}$. The external multiplier and the remaining terms give \[\frac{\mu_{\hat{l}-1,\hat{r}}+1}{(\hat{l}-1)!^{\mu_{\hat{l}-1,\hat{r}}+1}(\mu_{\hat{l}-1,\hat{r}}+1)!\hat{l}!^{\mu_{\hat{l},\hat{r}}-1}(\mu_{\hat{l},\hat{r}}-1)!}=
\frac{\hat{l}\mu_{\hat{l},\hat{r}}}{(\hat{l}-1)!^{\mu_{\hat{l}-1,\hat{r}}}\mu_{\hat{l}-1,\hat{r}}!\hat{l}!^{\mu_{\hat{l},\hat{r}}}\mu_{\hat{l},\hat{r}}!},\] so that the expression arising from such $\hat{l}$ and $\hat{r}$ is $\frac{\hat{l}\mu_{\hat{l},\hat{r}}}{n+1}C_{\beta}$ (recall that $C_{\beta}$ has $(n+1)!$ in the numerator in Definition \ref{combcoeff}, but here we only have $n!$). Moreover, the multiplier $\hat{l}\mu_{\hat{l},\hat{r}}$ vanishes wherever $\overline{\delta}_{\mu_{\hat{l},\hat{r}},0}\overline{\delta}_{\hat{l},0}$ does, so that it suffices to put the restriction $\hat{l}+\hat{r}\geq3$ (from $\overline{\delta}_{\hat{l}+\hat{r},2}$) and deduce that the total contribution of these terms is $\frac{C_{\beta}}{n+1}\sum_{\hat{l}+\hat{r}\geq3}\hat{l}\mu_{\hat{l},\hat{r}}$.

If $\mu_{2,0}\geq1$ then consider now $\hat{r}\geq1$ and $\hat{l}$ such that $\hat{l}+\hat{r}\geq2$, $\sqbinom{\hat{l}}{\hat{r}}\neq\sqbinom{1}{1}$, and $\mu_{\hat{l},\hat{r}}\geq1$ (from the $\overline{\delta}$-symbols again). Then $h\geq2$ as well, the element $\beta_{b}^{\hat{l},\hat{r}}$ is in $A_{n,h-1}$, and the contribution that we get is $(\hat{l}+1)(\mu_{\hat{l}+1,\hat{r}-1}+1)$ times $n!(h-2)!\big/\prod_{l+r\geq2}(l!r!)^{\mu_{b,l,r}^{\hat{l},\hat{r}}}\mu_{b,l,r}^{\hat{l},\hat{r}}!$. Once again the denominators arising from every vector $\sqbinom{l}{r}$ other than $\sqbinom{\hat{l}}{\hat{r}}$, $\sqbinom{\hat{l}+1}{\hat{r}-1}$, or $\sqbinom{2}{0}$ are the the same ones as in $C_{\beta}$, but the external multiplier and the remaining denominators combine to \[\frac{(\hat{l}+1)(\mu_{\hat{l}+1,\hat{r}-1}+1)}
{[(\hat{l}+1)!(\hat{r}-1)!]^{\mu_{\hat{l}+1,\hat{r}-1}+1}(\mu_{\hat{l}+1,\hat{r}-1}+1)![\hat{l}!\hat{r}!]^{\mu_{\hat{l},\hat{r}}-1}(\mu_{\hat{l},\hat{r}}-1)!2^{\mu_{2,0}-1}(\mu_{2,0}-1)!}.\] This gives us $2\mu_{2,0}\hat{r}\mu_{\hat{l},\hat{r}}$ over the denominators appearing in $C_{\beta}$ that are associated with these three values of $\sqbinom{l}{r}$, or in total $\frac{2\mu_{2,0}\hat{r}\mu_{\hat{l},\hat{r}}}{(n+1)(h-1)}C_{\beta}$ because here we also have only $(h-2)!$ in the numerator. As the numerator here vanishes when $\overline{\delta}_{\mu_{2,0},0}\overline{\delta}_{\hat{r},0}\overline{\delta}_{\mu_{\hat{l},\hat{r}},0}=0$, only the restriction $\sqbinom{\hat{l}}{\hat{r}}\neq\sqbinom{1}{1}$ from $\overline{\delta}_{\sqbinom{\hat{l}}{\hat{r}},\sqbinom{1}{1}}$ remains significant, and the total contribution here is $\frac{2\mu_{2,0}C_{\beta}}{(n+1)(h-1)}\sum_{\sqbinom{\hat{l}}{\hat{r}}\neq\sqbinom{1}{1}}\hat{r}\mu_{\hat{l},\hat{r}}$.

In the last two terms, appearing when $\mu_{1,1}\geq1$ (hence $h\geq2$ again), we have $2\mu_{2,0}$ or $h-1$ times $n!(h-2)!\big/\prod_{l+r\geq2}(l!r!)^{\mu_{d,l,r}}\mu_{d,l,r}!$, since $\beta_{d} \in A_{n,h-1}$. The denominators coming from $\sqbinom{l}{r}\neq\sqbinom{1}{1}$ are the ones appearing in $C_{\beta}$, and the external multiplier and the terms associated with the vector $\sqbinom{1}{1}$ combine to $\frac{2\mu_{2,0}}{(\mu_{1,1}-1)!}=\frac{2\mu_{2,0}\mu_{1,1}}{\mu_{1,1}!}$ and $\frac{h-1}{(\mu_{1,1}-1)!}=\frac{\mu_{1,1}(h-1)}{\mu_{1,1}!}$. Once again we have the denominator $(n+1)(h-1)$ under $C_{\beta}$ (since the numerator here is $(h-2)!$ once more), and as $\mu_{1,1}$ makes the multiplier $\overline{\delta}_{\mu_{1,1},0}$ redundant, these contributions are $\frac{2\mu_{2,0}\mu_{1,1}C_{\beta}}{(n+1)(h-1)}$ and $\frac{\mu_{1,1}C_{\beta}}{n+1}$ respectively.

The total expression that we therefore consider is \[\frac{C_{\beta}}{n+1}\sum_{\hat{l}+\hat{r}\geq3}\hat{l}\mu_{\hat{l},\hat{r}}+\frac{2\mu_{2,0}C_{\beta}}{(n+1)(h-1)}\sum_{\sqbinom{\hat{l}}{\hat{r}}\neq\sqbinom{1}{1}}\hat{r}\mu_{\hat{l},\hat{r}}+
\frac{2\mu_{2,0}\mu_{1,1}C_{\beta}}{(n+1)(h-1)}+\frac{\mu_{1,1}C_{\beta}}{n+1},\] which is a sum of four terms. The third term is precisely the summand associated with $\sqbinom{\hat{l}}{\hat{r}}=\sqbinom{1}{1}$, which is missing in the second term, and as $\sum_{\hat{l}+\hat{r}\geq2}\hat{r}\mu_{\hat{l},\hat{r}}=h-1$ by Definition \ref{Andef} and Remark \ref{setpart}, these two terms reduce to $\frac{2\mu_{2,0}C_{\beta}}{n+1}$. But this is $\hat{l}\mu_{\hat{l},\hat{r}}\frac{C_{\beta}}{n+1}$ with $\sqbinom{\hat{l}}{\hat{r}}=\sqbinom{2}{0}$, the fourth term is $\hat{l}\mu_{\hat{l},\hat{r}}\frac{C_{\beta}}{n+1}$ for $\sqbinom{\hat{l}}{\hat{r}}=\sqbinom{1}{1}$, the term $\hat{l}\mu_{\hat{l},\hat{r}}\frac{C_{\beta}}{n+1}$ associated with $\sqbinom{\hat{l}}{\hat{r}}=\sqbinom{0}{2}$ vanishes, and these are all the vectors $\sqbinom{\hat{l}}{\hat{r}}$ with $\hat{l}+\hat{r}=2$. The total expression is thus $\frac{C_{\beta}}{n+1}\sum_{\hat{l}+\hat{r}\geq2}\hat{l}\mu_{\hat{l},\hat{r}}$, which is just $C_{\beta}$ since $\beta \in A_{n+1}$ (Definition \ref{Andef} again). This proves the proposition.
\end{proof}

Our final result is now as follows.
\begin{thm}
Let $A_{n}$ be as in Definition \ref{Andef}, and for every element $\alpha$ in that set consider the constant $C_{\alpha}$ from Definition \ref{combcoeff}. Then the associated expression $\prod_{l+r\geq2}(\Delta_{l}f_{y^{r}})^{m_{l,r}}\big/f_{y}^{n+\sum_{l+r\geq2}m_{l,r}}$ from Corollary \ref{formwden} appears in $y^{(n)}$ with the coefficient $(-1)^{\sum_{l+r\geq2}m_{l,r}}C_{\alpha}$, namely we have
\[y^{(n)}=\sum_{\alpha \in A_{n}}\frac{(-1)^{\sum_{l+r\geq2}m_{l,r}}n!\big(\sum_{l+r\geq2}rm_{l,r}\big)!}{\prod_{l+r\geq2}l!^{m_{l,r}}r!^{m_{l,r}}m_{l,r}!}\cdot\frac{\prod_{l+r\geq2}(\Delta_{l}f_{y^{r}})^{m_{l,r}}}{f_{y}^{n+\sum_{l+r\geq2}m_{l,r}}}.\]
\label{coeffval}
\end{thm}

\begin{proof}
Corollary \ref{formwden} reduces us to proving that $c_{\alpha}=(-1)^{\sum_{l+r\geq2}m_{l,r}}C_{\alpha}$ for every $\alpha \in A_{n}$, a claim that we shall prove by induction on $n$. When $n=2$ we have the single element $\alpha$ with $m_{2,0}=1$ and $m_{l,r}=0$ for every other $l$ and $r$, for which it is easily seen in Definition \ref{combcoeff} that $C_{\alpha}=1$, and the claim is true since $c_{\alpha}=-1$ and $\sum_{l,r}m_{l,r}=1$. Assume now that the assertion holds for every $\alpha \in A_{n}$, take $\beta=\{\mu_{l,r}\}_{\{l+r\geq2\}} \in A_{n+1}$, and express $c_{\beta}$ as in Corollary \ref{contcbeta}. Assume again that $\beta \in A_{n+1,h}$, so that the proof of Proposition \ref{indcomb} implies that every well-defined element of the form $\beta_{-}^{\hat{l},\hat{r}}$ lies in $A_{n,h}$, while if $\beta_{b}^{\hat{l},\hat{r}}$ or $\beta_{d}$ are defined then they lie in $A_{n,h-1}$. Taking the signs into account, we thus have to show that the number $C_{\beta}$ from Definition \ref{combcoeff} can be expressed as in Corollary \ref{contcbeta}, but with each $c_{\alpha}$ replaced by $C_{\alpha}$ and the two external minus signs transformed into pluses. As this is precisely the content of Proposition \ref{indcomb}, this completes the proof of the theorem.
\end{proof}

We can compare Theorem \ref{coeffval} with the explicit expressions we already have for $n=3$ and $n=4$. One element of $A_{3}$ has $m_{3,0}=1$ and $m_{l,r}=0$ for every other $l$ and $r$, for which $\sum_{l,r}m_{l,r}=1$ and Definition \ref{combcoeff} gives $C_{\alpha}=1$, and indeed we had $c_{\alpha}=-1$. In the other element the multiplicities $m_{2,0}$ and $m_{1,1}$ are 1 and the rest vanish, so that $\sum_{l,r}m_{l,r}=2$, and the value $C_{\alpha}=3$ from Definition \ref{combcoeff} is again in correspondence with $c_{\alpha}$ being 3 as well. As for $A_{4}$, in $A_{4,1}$ there is the single element with the only non-vanishing multiplicity $m_{4,0}=1$, the set $A_{4,2}$ contains the element having $m_{3,0}=m_{1,1}=1$ and the element in which $m_{2,1}=m_{2,2}=1$, and the elements of $A_{4,3}$ are the one with $m_{2,0}=2$ and $m_{0,2}=1$ and the one with $m_{2,0}=1$ and $m_{1,1}=2$ (and all the other $m_{r,l}$'s vanish in each of them). Comparing the values of $h$ and the respective values $\frac{4!}{4!}=1$, $\frac{4!}{3!}=4$, $\frac{4!}{2!2!}=6$, $\frac{4!2!}{2!^{2}2!2!}=3$, and $\frac{4!2!}{2!2!}=12$ from Definition \ref{combcoeff} (omitting every $0!$ and $1!$) with the coefficients in Equation \eqref{yder4} verifies Theorem \ref{coeffval} for $n=4$ as well.

Since the coefficients $C_{\alpha}$ from Definition \ref{combcoeff} have a combinatorial meaning, let us review the proof of Proposition \ref{indcomb} (and with it of Theorem \ref{coeffval}) combinatorially. Assuming that $\beta=\{\mu_{l,r}\}_{l+r\geq2} \in A_{n+1,h}$, we need to count the partitions of $n+1$ blue balls and $h-1$ red balls into $h$ boxes such that the number of boxes containing $l$ blue balls and $r$ red balls is $\mu_{l,r}$, and recall that we only work with partitions in which boxes contain two balls (of any color) or more. First we consider those partitions in which the blue ball with number $n+1$ lies in a box containing at least two other balls, i.e., it comes from a box with $\hat{l}$ blue balls (and then $\hat{l}\geq1$ because our ball in question is there) and $\hat{r}$ red balls such that $\hat{l}+\hat{r}\geq3$. From such partitions we can simply take this ball out, yielding a partition that is associated with $\beta_{-}^{\hat{l},\hat{r}}$. On the other hand, given a partition of type $\beta_{-}^{\hat{l},\hat{r}}$, if we want to obtain a partition of type $\beta$ again we have to put the ball number $n+1$ into one of the boxes associated with $\sqbinom{\hat{l}-1}{\hat{r}}$, and there are $\mu_{-,\hat{l}-1,\hat{r}}^{\hat{l},\hat{r}}=\mu_{\hat{l}-1,\hat{r}}+1$ such boxes. Hence the first term in Proposition \ref{indcomb} counts partitions of type $\beta$ in which the blue ball of number $n+1$ lies in a box that has at least 3 balls in total.

Consider now those partitions in which the ball $n+1$ lies only with another red ball (which can happen only if $\mu_{1,1}\geq1$ and hence $h\geq2$). Its box-mate can be any of the $h-1$ red balls, so that choosing this ball gives us a multiplier of $h-1$, and by taking out the entire box we get a partition, now of $n$ blue balls and $h-2$ red balls, that is of type $\beta_{d}$. On the other hand, for every choice of partition of $\beta_{d}$, and every choice of a red ball, we can simply add a new box with the blue ball $n+1$ and the chosen red ball, and get a partition of type $\beta$. Hence the fourth term in Proposition \ref{indcomb} corresponds to partitions of type $\beta$ where $n+1$ has a single box-mate, which is red.

For counting the partitions having the blue ball $n+1$ with a single blue box-mate (which occur when $\mu_{2,0}\geq1$, implying also that $h\geq2$ since $\beta \in A_{n+1,h}$ for $n\geq2$) in terms of partitions coming from elements of $A_{n}$ we cannot simply take out the box containing $n+1$. What we do is find the red ball $h-1$ (which exists since $h\geq2$), replace it by the box-mate of $n+1$, and throw away the blue $n+1$, the red $h-1$, and the box. This red ball can be in a box with any parameters $\hat{r}\geq1$ (since the red ball $h-1$ is there) and $\hat{l}$ such that $\hat{l}+\hat{r}\geq2$, and note that our operation replaces this box by a box associated with $\sqbinom{\hat{l}+1}{\hat{r}-1}$ (which we also recall that contains the previous blue box-mate of $n+1$), and we took out a box associated with $\sqbinom{2}{0}$. This partition is therefore of type $\beta_{t}^{\hat{l},\hat{r}}$ when $\sqbinom{\hat{l}}{\hat{r}}\neq\sqbinom{1}{1}$, but of type $\beta_{d}$ again when $\sqbinom{\hat{l}}{\hat{r}}=\sqbinom{1}{1}$. Conversely, given a partition associated with $\beta_{t}^{\hat{l},\hat{r}}$ (when $\sqbinom{\hat{l}}{\hat{r}}\neq\sqbinom{1}{1}$), in order to create a partition of the form $\beta$ we need to choose the box of type $\sqbinom{\hat{l}+1}{\hat{r}-1}$ (of which we have $\mu_{b,\hat{l}+1,\hat{r}-1}^{\hat{l},\hat{r}}=\mu_{\hat{l}+1,\hat{r}-1}+1$) and one of the $\hat{l}+1$ blue balls that it contains, add the red ball $h-1$ to it, and take the chosen blue ball to be the box-mate of $n+1$ in an additional box. Doing this with a partition of type $\beta_{d}$ (so that $\sqbinom{\hat{l}}{\hat{r}}=\sqbinom{1}{1}$) works by the same argument, apart for the number of possible boxes now being $\mu_{d,2,0}=\mu_{2,0}$ (and with $\hat{l}+1=2$). Therefore the third term in Proposition \ref{indcomb} describes partitions of type $\beta$ where $n+1$ has a single, blue box-mate and the red ball $h-1$ also has a single blue box-mate, while the second term there counts those partitions in which $n+1$ still has a single blue box-mate but the red ball $h-1$ lies in a box of any other kind.

In total, the combinatorial explanation of Proposition \ref{indcomb} is that in partitions of type $\beta$ the blue ball can be either in a $\sqbinom{\hat{l}}{\hat{r}}$-box with $\hat{l}\geq1$ and $\hat{l}+\hat{r}\geq3$ (and $\mu_{\hat{l},\hat{r}}\geq1$), or in a $\sqbinom{1}{1}$-box when $\mu_{1,1}\geq1$, or in a $\sqbinom{2}{0}$-box $\mu_{2,0}\geq1$ (since only $\sqbinom{\hat{l}}{\hat{r}}$-boxes with $\hat{l}+\hat{r}\geq2$ are allowed and we must have $\hat{l}\geq1$ since the ball is blue). In the latter case, where $h$ must be at least 2, the red ball $h-1$ can either be in a $\sqbinom{1}{1}$-box (again, when $\mu_{1,1}\geq1$) or in any $\sqbinom{\hat{l}}{\hat{r}}$-box with $\sqbinom{\hat{l}}{\hat{r}}\neq\sqbinom{1}{1}$, $\mu_{\hat{l},\hat{r}}\geq1$, $\hat{l}+\hat{r}\geq2$, and $\hat{r}\geq1$ (since it is red). Moreover, these options are mutually exclusive, and counting each one of them leads to a multiple of the number of partitions of the appropriate type $\alpha \in A_{n}$ (which is $\beta_{-}^{\hat{l},\hat{r}}$, $\beta_{d}$, $\beta_{d}$ again, and $\beta_{b}^{\hat{l},\hat{r}}$, respectively). This is the combinatorial proof of Proposition \ref{indcomb} (and therefore also of Theorem \ref{coeffval}).

\section{The Meaning of the Expressions $\Delta_{l}f_{y^{r}}$ \label{Expansion}}

Let us consider the expression from Theorem \ref{coeffval} in the case where $f_{x}=0$. In this case only the term with $j=0$ in Definition \ref{Deltalg} survives, and every $\Delta_{l}f_{y^{r}}$ becomes just $f_{x^{l}y^{r}} \cdot f_{y}^{l}$. Recalling that $\sum_{l,r}lm_{l,r}=n$, we find that in this case the result of that theorem reduces to
\begin{equation}
y^{(n)}=\sum_{\alpha \in A_{n}}\frac{(-1)^{\sum_{l+r\geq2}m_{l,r}}n!\big(\sum_{l+r\geq2}rm_{l,r}\big)!}{\prod_{l+r\geq2}l!^{m_{l,r}}r!^{m_{l,r}}m_{l,r}!}\cdot\frac{\prod_{l+r\geq2}f_{x^{l}y^{r}}^{m_{l,r}}}{f_{y}^{\sum_{l+r\geq2}m_{l,r}}}. \label{fx=0}
\end{equation}
The first conclusion that we draw from that expression is that when $f_{x}=0$ our expression for $y^{(n)}$ coincides with that from \cite{[J3]}. Indeed, in this case only the terms from that reference whose corresponding partitions do not involve singletons remain (since the allowed singletons produce powers of $f_{x}$), and for these terms the coefficients there are precisely those appearing in our Definition \ref{combcoeff} (with our blue balls being the ``small numbers'' there, the red balls are the ``big numbers'', and the index $k$ is denoted here by $h$).

Assuming again that $f_{x}(x_{0},y_{0})$ is arbitrary, consider now the function $\varphi(x,z)$ defined to be $f(x,z+\lambda x)$ for some scalar $\lambda$. Differentiating with respect to $z$ and substituting $z=z_{0}=y_{0}-\lambda x_{0}$ shows that $\varphi_{z}(x_{0},y_{0}-\lambda x_{0})=f_{y}(x_{0},y_{0})\neq0$, so that the equality $\varphi(x,z)=0$ determines $z$ as a function of $x$ around $x_{0}$ as well (with $z(x_{0})=z_{0}=y_{0}-\lambda x_{0}$). Comparing with $y=y(x)$ arising from $f(x,y)=0$ implies that $y(x)=z(x)+\lambda (x-x_{0})$. Moreover, we find that $\varphi_{x}=f_{x}+\lambda f_{y}$, so that by setting $\lambda=y'=-\frac{f_{x}}{f_{y}}$ (again all the functions are evaluated at $x_{0}$ from now on) we find that $z'(x_{0})=0$. We can now prove the following relation.
\begin{lem}
With this value of $\lambda$ we have the equality $\varphi_{x^{l}z^{r}}=\Delta_{l}f_{y^{r}}/f_{y}^{l}$ for every $l$ and $r$. \label{trans}
\end{lem}

\begin{proof}
A very simple induction on $r$ shows that $\varphi_{z^{r}}(x,z)=f_{y^{r}}(x,y)$ for every $r$ when $y=z+\lambda x$. On the other hand, differentiating by $x$ gives (as we have seen above) that $\varphi_{xz^{r}}(x,z)=f_{xy^{r}}(x,y)+\lambda f_{y^{r+1}}$. The standard inductive argument (with the usual binomial identity) therefore proves that our expression $\varphi_{x^{l}z^{r}}$ is $\sum_{k=0}^{l}\binom{l}{k}\lambda^{k}f_{x^{l-k}y^{r+k}}$. As substituting $\lambda=-\frac{f_{x}}{f_{y}}$ produces the sum from Definition \ref{Deltalg} with $g=f_{y^{r}}$ divided by $f_{y}^{l}$, this proves the lemma.
\end{proof}

It therefore follows that the formula from Equation \eqref{fx=0} (which involves only powers of derivatives of $f$ in an elementary manner) in case the first derivative vanishes implies the general formula from Theorem \ref{coeffval}. Indeed, with the special value of $\lambda$ the function $z$ was seen to satisfy $z'(x_{0})=0$, and therefore $z^{(n)}$ can be evaluated as in Equation \eqref{fx=0} (with the derivatives of $f$ with respect to $y$ replaced by those of $\varphi$ with respect to $z$). But $y$ and $z$ differ by a linear function of $x$, so that $y^{(n)}=z^{(n)}$ for any $n\geq2$. By writing the expression $\varphi_{x^{l}z^{r}}$ appearing in Equation \eqref{fx=0} as in Lemma \ref{trans} (and in particular $\varphi_{z}$ from the denominator as $f_{y}$), and recalling that the total denominators from that lemma is $f_{y}^{n}$ by the known value of $\sum_{l,r}lm_{l,r}$, we recover the formula from Theorem \ref{coeffval}.

We would like to compare our formula with the one from \cite{[J3]} also when $f_{x}'\neq0$. Note that in that reference the formula for $y^{(n)}$ is given in terms of the more elementary products, so that in a notation similar to Corollary \ref{formwden} it is written as $\sum_{\gamma=\{s_{p,t}\}_{p,t} \in B_{n}}\Big[d_{\gamma}\prod_{p,t}f_{x^{p}y^{t}}^{s_{p,t}}\Big/f_{y}^{\sum_{p,t}s_{p,t}}\Big]$. Here $B_{n}$ is the set of partitions $\{s_{p,t}\}_{p,t}$ (with $p$ and $t$ non-negative integers) satisfying $\sum_{p,t}ps_{p,t}=n$ and $\sum_{p,t}(t-1)s_{p,t}=-1$ with $s_{0,0}=s_{0,1}=0$ (this set contains all the partitions coming from $A_{n}$ by adding $s_{1,0}=0$ to every element of that set, but usually contains also elements with $s_{1,0}>0$). As in Remark \ref{setpart}, this set is divided into the disjoint union of the sets $B_{n,k}$ according to the value $k$ of the sum $\sum_{p,t}s_{p,t}$, which here satisfies $1 \leq k \leq2n-1$. To see this, observe that the inequality $k\geq1$ is immediate once more, the sum $\sum_{p+t\geq2}(p+t)s_{p,t}$ equals $n+k-1-s_{1,0}$ and is bounded from below by $2\sum_{p+t\geq2}s_{p,t}=2(k-s_{1,0})$, and the second inequality follows since $s_{1,0}\leq\sum_{p,t}ps_{p,t}=n$. The coefficient $d_{\gamma}$ for $\gamma \in B_{n}$ equals $(-1)^{k}D_{\gamma}$ with $D_{\gamma}>0$ in case $\gamma \in B_{n,k}$, where the number $D_{\gamma}$ for $\gamma=\{s_{p,t}\}_{p,t} \in B=\bigcup_{n=1}^{\infty}B_{n}$ is described combinatorially in \cite{[J3]} and algebraically in \cite{[Wi]}. An argument similar to Definition \ref{combcoeff} shows that for $\gamma \in B_{n,k}$ we have
\begin{equation}
d_{\gamma}=(-1)^{k}D_{\gamma}=(-1)^{k}\Bigg(\sum_{p,t}ps_{p,t}\Bigg)!\Bigg(\sum_{p,t}ts_{p,t}\Bigg)!\Bigg/\prod_{p,t}p!^{s_{p,t}}t!^{s_{p,t}}s_{p,t}! \label{dgamma}
\end{equation}
Here the set $B_{1}=B_{1,1}$ is also defined, and it consists of the unique element in which $s_{1,0}=1$ and $s_{p,t}=0$ for any other $p$ and $t$, corresponding to the initial formula $y'=-\frac{f_{x}}{f_{y}}$ (since Equation \eqref{dgamma} yields $d_{\gamma}=-1$ there). In general, the result of \cite{[J3]} and \cite{[Wi]} is as follows.
\begin{thm}
Using the notation from the previous paragraph, we have
\[y^{(n)}=\sum_{k=1}^{2n-1}\sum_{\gamma=\{s_{p,t}\}_{p,t} \in B_{n,k}}\frac{(-1)^{k}D_{\gamma}}{f_{y}^{k}}\prod_{p,t}f_{x^{p}y^{t}}^{s_{p,t}},\] where $D_{\gamma}$ is given in Equation \eqref{dgamma}. \label{claselts}
\end{thm}

For the comparison, namely for proving Theorem \ref{claselts} from Theorem \ref{coeffval}, we introduce, for every element $\gamma=\{s_{p,t}\}_{p,t} \in B$, the set $Z_{\gamma}$ consisting of all those systems of numbers $\{q_{p,t,j}\}_{p+t\geq2,0 \leq j \leq t}$ such that $\sum_{j=0}^{t}q_{p,t,j}=s_{p,t}$ wherever $p+t\geq2$ and $\sum_{p+t\geq2}\sum_{j=0}^{t}jq_{p,t,j}=s_{1,0}$.
\begin{lem}
Expanding the formula from Theorem \ref{coeffval} in terms of the derivatives $f_{x^{p}y^{t}}$, and writing each $\gamma \in B_{n,k}$ as $\{s_{p,t}\}_{p,t}$, expresses $y^{(n)}$ as the sum of $1 \leq k \leq2n-1$ of \[\sum_{\gamma \in B_{n,k}}\Bigg[\frac{(-1)^{k}n!(k-s_{1,0}-1)!}{\prod_{p,t}(p!t!)^{s_{p,t}}}\times\sum_{\{q_{p,t,j}\}_{p,t,j} \in Z_{\gamma}}\prod_{j=0}^{t}\frac{1}{q_{p,t,j}!}\binom{t}{j}^{q_{p,t,j}}\Bigg]\frac{\prod_{p,t}f_{x^{p}y^{t}}^{s_{p,t}}}{f_{y}^{k}}.\] \label{expDeltalfyr}
\end{lem}

\begin{proof}
We expand each of the expressions $\Delta_{l}f_{y^{r}}$ appearing in Theorem \ref{coeffval} as in Definition \ref{Deltalg}, and take the $m_{l,r}$th power. The Multinomial Theorem produces, for every system of $l+1$ numbers $\{\tilde{q}_{l,r,j}\}_{j=0}^{l}$ with $\sum_{j=0}^{l}\tilde{q}_{l,r,j}=m_{l,r}$, the expression $\prod_{j=0}^{l}\big[(-1)^{j}\binom{l}{j}f_{x^{l-j}y^{r+j}} \cdot f_{x}^{j}f_{y}^{l-j}\big]^{\tilde{q}_{l,r,j}}$ multiplied by the multinomial coefficient $m_{l,r}!\big/\prod_{j=0}^{l}\tilde{q}_{l,r,j}!$. Assuming that the $m_{l,r}$'s come from an element $\alpha \in A_{n,h}$, we take the product over $l$ and $r$, and multiply by the coefficients $c_{\alpha}=(-1)^{h}C_{\alpha}$ from Theorem \ref{coeffval} and Definition \ref{combcoeff}. The powers of $l!$ cancel, the factors $m_{l,r}!$ cancel, and the terms arising from $f_{y}^{l}$ cancels with $f_{y}^{n}$ from the denominator, so that the summand associated with each system $\{\tilde{q}_{l,r,j}\}_{j=0}^{l}$ in the term corresponding to $\alpha$ is
\begin{equation}
\frac{(-1)^{h+\sum_{l,r,j}j\tilde{q}_{l,r,j}}n!(h-1)!f_{x}^{\sum_{l,r,j}j\tilde{q}_{l,r,j}}}{f_{y}^{h+\sum_{l,r,j}j\tilde{q}_{l,r,j}}}\prod_{l+r\geq2}\prod_{j=0}^{l}\frac{f_{x^{l-j}y^{r+j}}^{\tilde{q}_{l,r,j}}}
{\tilde{q}_{l,r,j}!(r!j!(l-j)!)^{\tilde{q}_{l,r,j}}}, \label{powmlrexp}
\end{equation}
where each occurrence of $\sum_{l,r,j}$ means $\sum_{l+r\geq2}\sum_{j=0}^{l}$ (we have also expanded the power $m_{l,r}=\sum_{j=0}^{l}\tilde{q}_{l,r,j}$ of $r!$). Since the parameters $m_{l,r}$ do not appear explicitly in Equation \eqref{powmlrexp}, the total contributions arising from all the elements of $A_{n,h}$ is the sum of the terms appearing in that equation over the set of those numbers $\{\tilde{q}_{l,r,j}\}_{l+r\geq2,0 \leq j \leq l}$ satisfying $\sum_{l,r,j}l\tilde{q}_{l,r,j}=n$, $\sum_{l,r,j}\tilde{q}_{l,r,j}=h$, and $\sum_{l,r,j}r\tilde{q}_{l,r,j}=h-1$.

We now separate the set of systems of numbers $\{\tilde{q}_{l,r,j}\}_{l,r,j}$ according to the value of the sum $s_{1,0}=\sum_{l,r,j}j\tilde{q}_{l,r,j}$, which lies between 0 and $\sum_{l,r,j}l\tilde{q}_{l,r,j}=n$, and take the external sum over $s_{1,0}$. Moreover, for the total formula we have to sum over $1 \leq h \leq n-1$ as well. In addition, we introduce the indices $p=l-j$ and $t=r+j$, so that $p+t=l+r\geq2$, the index $j$ satisfies $0 \leq j \leq t$, and we denote $\tilde{q}_{l,r,j}=\tilde{q}_{p+j,t-k,j}$ by simply $q_{p,t,j}$. The sums over $l$, $r$, and $j$ are the same as those over $p$, $t$, and $j$ (with $\sum_{p,t,j}$ similarly meaning $\sum_{p+t\geq2}\sum_{j=0}^{t}$), and the equalities appearing above for the numbers $q_{p,t,j}$ now become $\sum_{p,t,j}pq_{p,t,j}=n-s_{1,0}$, $\sum_{p,t,j}q_{p,t,j}=h$, and $\sum_{p,t,j}tq_{p,t,j}=h+s_{1,0}-1$. When we write $k=h+s_{1,0}$, which satisfies $1 \leq k\leq2n-1$, Equation \eqref{powmlrexp} for a given system of numbers $\{q_{p,t,j}\}_{p,t,j}$ satisfying these equalities takes the form \[\frac{(-1)^{k}n!(k-s_{1,0}-1)!f_{x}^{s_{1,0}}}{f_{y}^{k}}\prod_{p+t\geq2}\frac{f_{x^{p}y^{t}}^{\sum_{j=0}^{t}q_{p,t,j}}}{(p!t!)^{\sum_{j=0}^{t}q_{p,t,j}}}\prod_{j=0}^{t}\frac{1}{q_{p,t,j}!}\binom{t}{j}^{q_{p,t,j}}.\] For every $p$ and $t$ we set $s_{p,t}=\sum_{j=0}^{t}q_{p,t,j}$, and by adding the index $s_{1,0}$ from before the resulting equalities show that $\{s_{p,t}\}_{p,t}$ is an element of $B_{n,k}$. As the set of systems that contribute to a given element $\gamma \in B_{n,k}$ is precisely $Z_{\gamma}$, and the power of $p!t!$ is 1 for $p=1$ and $t=0$, this proves the lemma.
\end{proof}

Completing the comparison requires the following claim.
\begin{prop}
Let $\{s_{p,t}\}_{p+t\geq2}$ be a set of non-negative integers, only finitely many of which are nonzero and such that $\sum_{p+t\geq2}ts_{p,t}=k-1$, and let $s_{1,0}$ be any integer between 0 and $k-1$. If $\gamma$ denotes the full set $\{s_{p,t}\}_{p,t}$ then the sum $\sum_{\{q_{p,t,j}\}_{p,t,j} \in Z_{\gamma}}\prod_{p+t\geq2}s_{p,t}!\prod_{j=0}^{t}\binom{t}{j}^{q_{p,t,j}}\big/q_{p,t,j}!$ equals $\binom{k-1}{s_{1,0}}$. \label{compJohnson}
\end{prop}

An algebraic proof of Proposition \ref{compJohnson} seems difficult in general, but we can give a combinatorial one.
\begin{proof}
Assume that $k-1$ balls are given, out of which $ts_{p,t}$ carry the indices $p$ and $t$ wherever $p+t\geq2$ (so that the sum $\sum_{p,t}ts_{p,t}$ is indeed $k-1$), and they are held in $s_{p,t}$ boxes of $t$ balls each. The number of possible ways to select $s_{1,0}$ balls in total is classically known to be $\binom{k-1}{s_{1,0}}$. On the other hand, given such a choice and indices $p$ and $t$, consider the number $q_{p,t,j}$ of boxes with those indices from which precisely $j$ balls are chosen. These are meaningful only for $0 \leq j \leq t$ (since there are $t$ balls in every such box), and the sum $\sum_{j=0}^{t}q_{p,t,j}$ equals the number $s_{p,t}$ of such boxes. With these parameters there are $s_{p,t}!\big/\prod_{j=0}^{t}q_{p,t,j}!$ possibilities for deciding how many balls are taken from each of the $s_{p,t}$ boxes, and once this is determined, taking $j$ balls from each of the $q_{p,t,j}$ boxes can be done in $\binom{t}{j}^{q_{p,t,j}}$ ways (and we take the product over $0 \leq j \leq t$). Since the total number of balls is $s_{1,0}$ we also get the equality $\sum_{p,t,j}jq_{p,t,j}=s_{1,0}$, so that $\{q_{p,t,j}\}_{p,t,j}$ is indeed an element of the appropriate set $Z_{\gamma}$, and the contribution of that element to $\binom{k-1}{s_{1,0}}$ is the asserted one. Since every element of $Z_{\gamma}$ contributes in this way, and we have seen that these are the only ones, this proves the proposition.
\end{proof}

In total, Theorem \ref{coeffval}, Lemma \ref{expDeltalfyr}, and Proposition \ref{compJohnson} prove Theorem \ref{claselts}. Indeed, when we consider the coefficient associated with $\gamma=\{s_{p,t}\}_{p,t} \in B_{n,k}$ in Lemma \ref{expDeltalfyr}, we get $(-1)^{k}n!$ from the numerator and the powers of $p!$ and $t!$ from the denominator of Equation \eqref{dgamma}, and after extending our definition of $Z_{\gamma}$ for any $\gamma=\{s_{p,t}\}_{p,t}$ with non-negative $s_{p,t}$'s only finitely many of which are non-zero (but still with $s_{0,0}=s_{0,1}=0$), Proposition \ref{compJohnson} shows that the remainder of the coefficient is the one required for obtaining $D_{\gamma}$. Hence our results agree with \cite{[Wi]} and \cite{[J3]}. Moreover, the terms appearing in Theorem \ref{claselts} all involve $n$ differentiations with respect to $x$ in total, and the power of $f_{y}$ in the denominator is always 1 more than the total number of differentiations with respect to $y$, as \cite{[N]} predicts. This is also visible in our Proposition \ref{formnoden} or Corollary \ref{formwden}, when one observes in Definition \ref{Deltalg} that every term in $\Delta_{l}f_{y^{r}}$ contains $l$ differentiations with respect to $x$ and $l+r$ differentiations with respect to $y$.

\noindent\textsc{Einstein Institute of Mathematics, the Hebrew University of Jerusalem, Edmund Safra Campus, Jerusalem 91904, Israel}

\noindent E-mail address: zemels@math.huji.ac.il


\begin{thebibliography}{}{}

\bibitem[C1]{[C1]} Comtet, L., \textsc{Polyn\^{o}mes de Bell et Formule Explicite des D\'{e}riv\'{e}es Successives d’une Fonction Implicite}, C.R. Acad. Sc. Paris, S\'{e}rie A tome 267, 457--460 (1968).
\bibitem[C2]{[C2]} Comtet, L., \textsc{Advanced Combinatorics---The Art of Finite and Infinite Expansions}, D. Reidel Publishing Company/Springer Netherlands, Dordrecht, Holland, xi+343pp (1974).
\bibitem[CF]{[CF]} Comtet, L., Fiolet, M., \textsc{Sue les D\'{e}riv\'{e}es Successives d’une Fonction Implicite}, C.R. Acad. Sc. Paris, S\'{e}rie A tome 278, 249--251 (1974).
\bibitem[FGV]{[FGV]} Figueroa, H., Garcia--Bond\'{i}a, J. M., V\'{a}rilly, J. C., \textsc{Fa\`{a} di Bruno Hopf Algebras}, pre-print. arXiv link: https://arxiv.org/abs/math/0508337.
\bibitem[J1]{[J1]} Johnson, W. P., \textsc{The Curious History of Fa\`{a} di Bruno's Formula}, Am. Math. Monthly, vol 109 no. 3, 217--234 (2002).
\bibitem[J2]{[J2]} Johnson, W. P., \textsc{Combinatorics of Higher Derivatives of Inverses}, Am. Math. Monthly, vol 109 no. 3, 273-–277 (2002).
\bibitem[J3]{[J3]} Johnson, W. P., \textsc{Some Problems in Differentiation}, pre-print.
\bibitem[N]{[N]} Nahay, J. M., \textsc{The $n$th Order Implicit Differentiation Formula For Two Variables with an Application to Computing All Roots of a Transcendental Function}, Math. Comput. Sci., vol 6 issue 1, 79-–105 (2012).
\bibitem[STZ]{[STZ]} Schlank, T. M., Tessler, R. J., Zernik, A., \textsc{Exact Maximum-Entropy Estimate for Feynman Diagrams}, pre-print. arXiv link: https://arxiv.org/abs/1512.00752
\bibitem[S]{[S]} Sokal, A. D., \textsc{A Ridiculously Simple and Explicit Implicit Function Theorem}, S\'{e}minaire Lotharingien de Combinatoire, vol 61A, 21pp (2009).
\bibitem[Wi]{[Wi]} Wilde, T., \textsc{Implicit Higher Derivatives, and a Formula of Comtet and Fiolet}, pre-print, https://arxiv.org/pdf/0805.2674 (2008).
\bibitem[Wo]{[Wo]} Worontzoff, M., \textsc{Sur le D\'{e}veloppement en S\'{e}ries des Fonctions Implicites}, Nouv. Ann. Math., S\'{e}rie 3 tome 13, 167--184 (1894).
\bibitem[Y]{[Y]} Yuzhakov, A. P., \textsc{On an Application of the Multiple Logarithmic Residue to the Expansion of Implicit Functions in Power Series}, Mat. Sbornik, vol 92 no. 2, 177--192 (1975).

\end{thebibliography}
\end{document}